\documentclass[10pt]{amsart}
\usepackage{amsmath}
\usepackage{cases}
\usepackage{mathrsfs}
\usepackage{bbm}
\usepackage{amssymb}
\usepackage{amscd}
\usepackage{amsfonts,latexsym,amsmath,amsthm,amsxtra,mathdots,amssymb,latexsym,mathabx}
\usepackage[all,cmtip]{xy}
\RequirePackage{amsmath} \RequirePackage{amssymb}
\usepackage{color}
\usepackage{colordvi}
\usepackage{multicol}
\usepackage{hyperref}
\usepackage{mathtools}
\usepackage[margin=1.1in]{geometry}
\usepackage{xcolor}
\hypersetup{
    colorlinks,
    linkcolor={red!50!black},
    citecolor={blue!50!black},
    urlcolor={blue!80!black}
}

\newcommand{\sumstar}{\sideset{}{^\star}\sum}
\newcommand\sumsum{\mathop{\sum\sum}\limits}

    \newtheorem{thm}{Theorem}[section] 
    \newtheorem{lem}[thm]{Lemma}  \newtheorem{prop}[thm]{Proposition}
    
     \newtheorem{defn}[thm]{Definition}
    \newtheorem {rem}[thm]{Remark}

\begin{document}
\title{Bounds for twists of $\rm GL(3)$ {$L$}-functions}
\author[Y. Lin]{Yongxiao Lin}

\address{Department of Mathematics, The Ohio State University\\ 231 W 18th Avenue\\
Columbus, Ohio 43210-1174}
\email{lin.1765@buckeyemail.osu.edu}


\begin{abstract}
Let $\pi$ be a fixed Hecke--Maass cusp form for $\mathrm{SL}(3,\mathbb{Z})$ and $\chi$ be a primitive Dirichlet character modulo $M$, which we assume to be a prime. Let $L(s,\pi\otimes \chi)$ be the $L$-function associated to $\pi\otimes \chi$. In this paper, for any given $\varepsilon>0$, we establish a subconvex bound $L(1/2+it, \pi\otimes \chi)\ll_{\pi, \varepsilon} (M(|t|+1))^{3/4-1/36+\varepsilon}$, uniformly in both the $M$- and $t$-aspects. 
\end{abstract}



\maketitle

\section{Introduction and statement of results}
The subconvexity problem, which asks for an estimate of an automorphic $L$-function on the critical line $s=1/2+it$ that is better by a power saving than the bound implied by the functional equation and the Phragmen--Lindel\"of principle, is one of the central problems in analytic number theory. Many cases have been treated in the past; see \cite{Michel-Venkatesh} for results with full generality on $\rm GL(2)$. It has only been recently that people have started making progress on $\rm GL(3)$ with the introduction of new techniques.

In this paper, we are interested in certain degree $3$ $L$-functions. Let $\pi$ be a fixed Hecke--Maass cusp form of type $(\nu_1,\nu_2)$ for $\mathrm{SL}(3,\mathbb{Z})$ with normalized Fourier coefficients $\lambda(m,n)$. Let $\chi$ be a primitive Dirichlet character modulo $M$. Let
\begin{equation*}
\begin{split}
L(s,\pi)=\sum_{n=1}^{\infty}\frac{\lambda(1,n)}{n^s}\quad \mbox{and}\quad 
L(s,\pi\otimes \chi)=\sum_{n=1}^{\infty}\frac{\lambda(1,n)\chi(n)}{n^s}
\end{split}\end{equation*}
be the $L$-series associated with $\pi$ and $\pi\otimes \chi$; these series can be continued to entire functions of $s\in \mathbb{C}$ with functional equations. The Phragmen--Lindel\"of principle implies the convexity bound  $L(1/2+it,\pi\otimes \chi)\ll_{\pi,\varepsilon} (M(|t|+1))^{3/4+\varepsilon}$, for which one aims to improve.

For the $L$-function $L(s,\pi)$, the first breakthrough was made by Li \cite{Li1} who resolved the subconvexity problem in the $t$-aspect for self-dual cusp form $\pi$. Using a first moment method, similar to the approach in \cite{Conrey-Iwaniec}, Li showed that $L(1/2+it,\pi)\ll_{\pi} (|t|+1)^{3/4-\delta+\varepsilon}$, with $\delta=1/16$. Li's approach also implies a subconvexity bound for certain $\rm GL(3)\times \rm GL(2)$ $L$-functions. The method depends on the non-negativity of central values of certain $L$-functions, which necessitates in the self-duality assumption on the cusp form $\pi$. Li's exponent of saving $\delta=1/16$ was later improved to $\delta=1/12$ by Mckee, Sun, and Ye \cite{MSY}, and to $\delta=1/8$ by Nunes \cite{Nunes}.

For the case where $M$, the conductor of the Dirichlet character $\chi$, is varying, in the special case that $\pi$ is self-dual and $\chi$ is quadratic, a subconvex bound was obtained by Blomer \cite{Blo12}. He showed that $L(1/2,\pi\otimes \chi)\ll_{\pi} M^{5/8+\varepsilon}$ by using a first moment method as in Li's work, where $M$ is assumed to be prime.  Later Huang \cite{Huang1}, with input from \cite{Young17}, managed to extend the results of Li and Blomer to the hybrid setting $L(1/2+it,\pi\otimes \chi)\ll_{\pi} (M(|t|+1))^{3/4-\delta}$, for some $\delta>0$, under the same self-duality assumptions on $\pi$ and $\chi$.

From a theorem of Miller \cite{Miller01}, self-dual cusp forms $f_j$ on $\rm SL(3,\mathbb{Z})\backslash \mathfrak{h}^3$ are sparse in the sense that 
\begin{equation*}
\lim_{T\rightarrow \infty}\frac{\#\{\lambda_j\leq T| \Delta f_j=\lambda_j f_j, f_j\, \text{self-dual}\}}{\#\{\lambda_j\leq T\}}=0.
\end{equation*}
It is therefore desirable to remove the self-duality assumptions in the previous works.

In a series of papers \cite{Munshi2014, Mun3, Mun4, Mun1, Mun2}, Munshi proposed a new approach to the subconvexity problem. This method does not need to assume the non-negativity of central values of certain $L$-functions, which enables Munshi to deal with more general cusp forms than just the self-dual subclass. 

In the $t$-aspect setting, by adopting Kloosterman's refinement of the circle method and enhanced by a ``conductor lowering" mechanism, Munshi \cite{Mun4} obtained the bound ${L(1/2+it,\pi)\ll_{\pi} (|t|+1)^{3/4-1/16+\varepsilon}}$, thus extending Li's result \cite{Li1} to arbitrary fixed cusp forms $\pi$. 

In the Dirichlet character twist case, by using a variant of the $\delta$-symbol method of Duke, Friedlander, and Iwaniec \cite{DFI}, a $\rm GL(2)$ Petersson $\delta$-symbol method, Munshi established $L\left(1/2,\pi\otimes \chi\right)\ll_{\pi} M^{3/4-1/1612+\varepsilon}$ \cite{Mun1}, under the Ramanujan conjecture for $\pi$. In a follow-up preprint \cite{Mun2}, with a much cleaner treatment, he removed such an assumption and improved the exponent of saving to $\delta=1/308$. Again, this approach does not require non-negativity of central values of certain $L$-functions, thereby removing the self-duality assumptions on the cusp forms $\pi$ and characters $\chi$ in Blomer's work \cite{Blo12}. 

More recently Holowinsky and Nelson \cite{HN17} discovered that there is a hidden identity within the proof of \cite{Mun2}, which allowed them to produce a method that removes the use of the Petersson $\delta$-symbol method and also improves the exponent of saving. They obtained a stronger subconvex exponent $L\left(1/2,\pi\otimes \chi\right)\ll_{\pi} M^{3/4-1/36+\varepsilon}$.

It is now desirable to ask, ``Can one prove a subconvex bound for the $L$-functions $L\left(s,\pi\otimes \chi\right)$, simultaneously in both the $M$- and $t$-aspects, for general $\rm SL(3,\mathbb{Z})$ Hecke--Maass cusp forms and primitive Dirichlet characters?" Our main result answers this affirmatively.
\begin{thm}\label{Main theorem}
Let $\pi$ be a Hecke--Maass cusp form for $\mathrm{SL}(3,\mathbb{Z})$ and $\chi$ be a primitive Dirichlet character modulo $M$, which we assume to be prime. Given any $\varepsilon>0$, we have
\begin{equation}\label{main saving}
\begin{split}
L\left(\frac{1}{2}+it,\pi\otimes \chi \right)\ll_{\pi, \varepsilon}&(M(|t|+1))^{3/4-1/36+\varepsilon}.
\end{split}
\end{equation}
\end{thm}

\begin{rem}
Below we will carry out the proof under the assumption $|t|>M^{\varepsilon}$ for any $\varepsilon>0$. We make such an assumption so as to control the error term of the stationary phase analysis in our approach. For the case where $|t|<M^{\varepsilon}$, the bound \eqref{main saving} would follow from the work \cite{HN17}, since there their bound $L\left(1/2+it,\pi\otimes \chi\right)\ll_{t, \pi} M^{3/4-1/36+\varepsilon}$ is of polynomial dependence in $t$.
%
\end{rem}


For subconvexity bounds on $\rm GL(3)$ in other aspects, see \cite{Blomer-Buttcane, Sun2017, Sun-Zhao17, Blomer-Buttcane18}.

Our approach is a variant of the methods introduced in the works \cite{Mun2} and \cite{HN17}. In Section \ref{An outline of the proof}, we will give a brief outline of our approach to guide the readers through.  Recently, Raphael Schumacher \cite{Schumacher} has been able to provide another interpretation of the methods we follow, at least in the $t$-aspect case, from the perspective of integral representations under the framework of Michel--Venkatesh \cite{Michel-Venkatesh}, and produces the same bound \eqref{main saving}.

\subsection*{Notation.} We use $e(x)$ to denote $\exp(2\pi ix)$. We denote $\varepsilon$ an arbitrary small positive constant, which might change from line to line.
In this paper the notation $A\asymp B$ (sometimes even $A\approx B$) means that $B/(M|t|)^{\varepsilon}\ll |A|\ll B(M|t|)^{\varepsilon}$. We reserve the letters $p$ and $\ell$ to denote primes. The notations $p\sim P$ and $\ell\sim L$ denote primes in the dyadic segments $[P,2P]$ and $[L,2L]$ respectively. 

\section{An outline of the proof}\label{An outline of the proof}
Our approach is inspired by the work \cite{Mun2} and makes use of an observation due to Holowinsky and Nelson \cite{HN17}. We now give a brief introduction to the approach in \cite{Mun2}. 

Let $p$ be a prime number, and let $k\equiv 3\bmod{4}$ be a positive integer. Let $\psi$ be a character of $\mathbb{F}_p^{\times}$ satisfying $\psi(-1)=-1=(-1)^k$. One can view $\psi$ as a character modulo $pM$. Let $H_k(pM,\psi)$ be an orthogonal Hecke basis of the space of cusp forms $S_k(pM,\psi)$ of level $pM$, nebentypus $\psi$ and weight $k$. For $f\in H_k(pM,\psi)$, let $(\lambda_f(n))_{n\geq 1}$ denote its Fourier coefficients. 
Denote $P^{\star}=\sum_{P<p<2P}\sum_{\psi\bmod{p}}  (1-\psi(-1))$. Then we have the following averaged version of the Petersson formula:
\begin{equation}\label{Petersson}
\begin{split}
\delta(r,n) =  \frac{1}{P^{\star}} \sum_{p\sim P}\sum_{\psi\bmod{p}} & (1-\psi(-1))\sum_{f\in H_k(pM,\psi)}\omega_f^{-1} \overline{\lambda_f(r)} \lambda_f(n)\\ & - \frac{2\pi i}{P^{\star}} \sum_{p\sim P}\sum_{\psi\bmod{p}} (1-\psi(-1)) \sum_{c=1}^{\infty}\frac{S_{\psi}(r,n;cpM)}{cpM} J_{k-1}\left(\frac{4\pi\sqrt{rn}}{cpM}\right),
\end{split}
\end{equation}
where $\delta(r,n)$ denotes the Kronecker symbol, $\omega_f^{-1}=\frac{\Gamma(k-1)}{(4\pi)^{k-1}\|f\|^2}$ is the spectral weight, and $S_{\psi}(r,n;c)=\sumstar_{\alpha \bmod{c}}\psi(\alpha)e\left(\frac{r\alpha+n\bar{\alpha}}{c}\right)$ is the generalized Kloosterman sum.

Let $\mathcal{L}$ be the set of primes in the interval $[L,2L]$ and let $L^{\star}=|\mathcal{L}|$ denote the cardinality of $\mathcal{L}$. By writing his main sum of interest $\mathop{\sum\sum}_{m,n=1}^{\infty}\lambda(m,n)\chi(n)W\left(\frac{nm^2}{N}\right)V\left(\frac{n}{N}\right)$ as
\begin{equation*}
\frac{1}{L^\star}\sum_{\ell\in\mathcal{L}}\bar{\chi}(\ell)\mathop{\sum\sum}_{m,n=1}^{\infty}\lambda(m,n)W\left(\frac{nm^2}{N}\right)\sum_{r=1}^{\infty}\chi(r)V\left(\frac{r}{N\ell}\right)\delta(r,n\ell),
\end{equation*}
and then substituting the formula \eqref{Petersson} for $\delta(r,n\ell)$ inside, Munshi breaks the main sum into two pieces $\mathcal{F}^\star+\mathcal{O}^\star$, with $\mathcal{F}^\star$ and $\mathcal{O}^\star$ appropriately defined. Here the introduction of the extra summation over $\ell$ serves the role of an amplification technique. Successfully bounding $\mathcal{F}^\star$ and $\mathcal{O}^\star$ simultaneously with suitable choices of $P$ and $L$ to balance the contribution enables Munshi to get his main result $L\left(1/2,\pi\otimes \chi\right)\ll_{\pi} M^{3/4-1/308+\varepsilon}$.  

Now we turn to our case. From Lemma \ref{approximate FE}, it suffices to improve the trivial bound $O(N^{1+\varepsilon})$ for the smooth sum
\begin{equation*}
S(N):=\sum_{n\geq1} \lambda(1,n) \chi(n)n^{-it}w\left(\frac{n}{N}\right),
\end{equation*}
 for $(Mt)^{3/2-\delta}<N<(Mt)^{3/2+\varepsilon}$, where $w\left(x\right)$ is some smooth function with compact support contained in $\mathbb{R}_{>0}$ satisfying $w^{(j)}(x)\ll 1$ for all $j\geq 0$. 
 
 For the purpose of this sketch we assume the Ramanujan bound $|\lambda(1,n)|\ll n^{\varepsilon}$.
 
Let $P$ and $L$ be two large parameters to be specified later. In our case, instead of using the Petersson $\delta$-symbol method \eqref{Petersson}, we use a ``key identity" \eqref{lin analogue2}, 
\begin{equation*}
\begin{split}
\chi(n)n^{-it}&V_A\left(\frac{n}{N}\right)\\
=&\bigg(\frac{2\pi}{Mt}\bigg)^{it} e\bigg(\frac{t}{2\pi}\bigg)\frac{M^{2}t^{3/2}\ell}{Np g_{\bar{\chi}}}\sum_{r=1}^{\infty}\chi(r\ell \bar{p})\left(\frac{r\ell}{p}\right)^{-it}e\bigg(-\frac{np\bar{M}}{\ell r}\bigg)V\left(\frac{r}{Np/M\ell t}\right)\\
&-\bigg(\frac{2\pi}{N}\bigg)^{it}e\bigg(\frac{t}{2\pi}\bigg) \frac{t^{1/2}}{g_{\bar{\chi}}}\sum_{r\neq 0}S_{\bar{\chi}}(n,rp\bar{\ell};M)\mathcal{J}_{it}(n,rp/\ell;M)+O\left(t^{1/2-A}\right),
\end{split}
\end{equation*}
where
\begin{equation*}
\mathcal{J}_{it}(n,rp/\ell;M)=\int_{\mathbb{R}}V(x)x^{-it}e\bigg(-\frac{nt}{Nx}\bigg)e\bigg(-\frac{rNpx}{M^2\ell t}\bigg)\,\mathrm{d}x.
\end{equation*}
Here $V\left(x\right)$ is a smooth compactly supported function satisfying $V^{(j)}(x)\ll 1$ for all $j\geq 0$.

Thus we can write, for an arbitrarily large $A\geq 1$, that
 \begin{equation*}
S(N)\asymp |\mathcal{F}|+|\mathcal{O}|+O\left(Nt^{-A}\right),
\end{equation*}
where
\begin{equation}\label{F-term for sketch}
\begin{split}
\mathcal{F}=&\frac{M^{3/2}t^{3/2}}{NP^2}\sum_{p\sim P}\sum_{\ell\sim L}\sum_{r\sim NP/MLt}\chi(r\ell\bar{p})\left(\frac{r\ell}{p}\right)^{-it}\sum_{n=1}^{\infty}\lambda(1,n)e\left(-\frac{np\bar{M}}{\ell r}\right)w\left(\frac{n}{N}\right),
\end{split}
\end{equation}
and
\begin{equation*}
\begin{split}
\mathcal{O}=&\frac{t^{1/2}}{PLM^{1/2}}\sum_{n=1}^{\infty}\lambda(1,n)w\left(\frac{n}{N}\right)\sum_{p\sim P}\sum_{\ell\sim L}\sum_{0\neq |r|\ll \frac{M^2t^2L}{NP}}S_{\bar{\chi}}(n,rp\bar{\ell};M)\mathcal{J}_{it}(n,rp/\ell;M).
\end{split}
\end{equation*}

Now our task is to beat the bound $O(N^{1+\varepsilon})$ for $\mathcal{F}$ and $\mathcal{O}$ simultaneously.
We estimate the term $\mathcal{O}$ first. The integral
$\mathcal{J}_{it}(n,rp/\ell;M)$ restricts the length of the $r$-sum to $0\neq |r|\ll \frac{M^2t^2L}{NP}$. For this sketch we pretend that $r\sim \frac{M^2t^2L}{NP}$.

From the second derivative test we have $\mathcal{J}_{it}(n,rp/\ell;M)\ll t^{-1/2}$. Using this, along with the Weil bound for Kloosterman sums, and estimating trivially, we find that
\begin{equation*}
\begin{split}
\mathcal{O}\ll \frac{t^{1/2}}{PLM^{1/2}}NPL
\frac{M^2t^2L}{NP}M^{1/2}t^{-1/2}
\ll  N\frac{M^2t^2L}{NP}.
\end{split}
\end{equation*}
So we need to save more than $\frac{M^2t^2L}{NP}$ for $\mathcal{O}$.

We apply the Cauchy--Schwarz inequality to reduce the task to saving the same amount from 
\begin{equation*}
\begin{split}
\frac{N^{1/2}t^{1/2}}{PLM^{1/2}}\left(\sum_{n\sim N}\bigg|\sum_{p\sim P}\sum_{\ell\sim L}
\sum_{r\sim \frac{M^2t^2L}{NP}}S_{\bar{\chi}}(n,rp\bar{\ell};M)\mathcal{J}_{it}(n,rp/\ell;M)\bigg|^2\right)^{1/2},
\end{split}
\end{equation*}
or equivalently, saving $\frac{M^4t^4L^2}{N^2P^2}$ from the sum
\begin{equation*}
\begin{split}
\sum_{n\sim N}\bigg|\sum_{p\sim P}\sum_{\ell\sim L}
\sum_{r\sim \frac{M^2t^2L}{NP}}S_{\bar{\chi}}(n,rp\bar{\ell};M)\mathcal{J}_{it}(n,rp/\ell;M)\bigg|^2.
\end{split}
\end{equation*}

For the diagonal term $(p_1,\ell_1,r_1)=(p_2,\ell_2,r_2)$, we save $PL\frac{M^2t^2L}{NP}=\frac{M^2t^2L^2}{N}$, which is satisfactory as long as $\frac{M^2t^2L^2}{N}>\frac{M^4t^4L^2}{N^2P^2}$; i.e., $P\gg \frac{Mt}{N^{1/2}}$.
 
Opening the square above and applying Poisson summation to the $n$-sum, only the zero frequency contributes.
For the off-diagonal $(p_1,\ell_1,r_1)\neq (p_2,\ell_2,r_2)$, applying Poisson summation in the $n$-sum we save $M$ from evaluating 
\begin{equation*}\sum_{a(M)}S_{\bar{\chi}}(a,r_1p_1\bar{\ell_1};M)\overline{S_{\bar{\chi}}(a,r_2p_2\bar{\ell_2};M)},
\end{equation*}
and save $t$ from estimating the integral
\begin{equation*}
\mathfrak{J}=\int_{\mathbb{R}}\mathcal{J}_{it}(Ny,r_1p_1/\ell_1;M)\overline{\mathcal{J}_{it}(Ny,r_2p_2/\ell_2;M)}\,w\left(y\right)\mathrm{d}y
\end{equation*}
upon using the first derivative test for oscillatory integral (which is the content of Lemma \ref{bound of J}). So the estimates for the off-diagonal are satisfactory so long as $Mt\gg \frac{M^4t^4L^2}{N^2P^2}$, i.e., $P>\frac{M^{3/2}t^{3/2}L}{N}$. Hence $\mathcal{O}$ is fine for our purpose if $P>\max\{\frac{Mt}{N^{1/2}},\frac{M^{3/2}t^{3/2}L}{N}\}$.

Next, we try to bound the $\mathcal{F}$ term in \eqref{F-term for sketch}. Estimating trivially, we see that
\begin{equation*}
\begin{split}
\mathcal{F}\ll \frac{M^{3/2}t^{3/2}}{NP^2}PL\frac{NP}{MLt}N
\ll N(Mt)^{1/2}.
\end{split}
\end{equation*}
So our job is to save more than $(Mt)^{1/2}$.

We apply Voronoi summation (Lemma \ref{voronoi}) to the $n$-sum to get
\begin{equation*}\label{voronoi for the case of t}
\begin{split}
\mathcal{F}\asymp\frac{M^{3/2}t^{3/2}N^{1/2}}{NP^2}\left|\sum_{p\sim P}\sum_{\ell\sim L}\sum_{r\sim NP/MLt}\chi(r\ell\bar{p})\left(\frac{r\ell}{p}\right)^{-it}\sum_{n\sim (\ell r)^3/N}\frac{\lambda(n,1)}{\sqrt{n}}\frac{S(\bar{p}M,n;r\ell)}{\sqrt{r\ell}}\right|.
\end{split}
\end{equation*}
Using the Weil bound and estimating trivially we get $\mathcal{F}\ll (NP)^{3/2}/Mt$. We save $\frac{M^{3/2}t^{3/2}}{N^{1/2}P^{3/2}}$ from this process, compared to the original trivial bound $N(Mt)^{1/2}$, and we still need to save a little more than $\frac{N^{1/2}P^{3/2}}{Mt}$ from the new sum. Pulling the $r,n$-sums outside, and applying the Cauchy--Schwarz inequality, our job is to save $\frac{NP^3}{M^2t^2}$ from the sum
\begin{equation*}
\begin{split}
\sum_{r\sim \frac{NP}{MLt}}\sum_{n\sim \frac{N^2P^3}{M^3t^3}}\bigg|\sum_{p\sim P}\sum_{\ell\sim L}\chi(\ell\bar{p})(\ell/p)^{-it}S(\bar{p}M,n;r\ell)\bigg|^2.
\end{split}
\end{equation*}
We can save $PL$ from the diagonal, which is satisfactory if $PL>\frac{NP^3}{M^2t^2}$, that is, $L>\frac{NP^2}{M^2t^2}$.
Our final step involves opening the square and applying Poisson summation to the $n$-sum to gain saving for the off-diagonal terms $(p_1,\ell_1)\neq (p_2,\ell_2)$.  The zero frequency (which vanishes unless $\ell_1=\ell_2$) makes a contribution that is dominated by the diagonal $(p_1,\ell_1)= (p_2,\ell_2)$ contribution. The original $n$-sum can be estimated by
\begin{equation*}
\sum_{n\sim \frac{N^2P^3}{M^3t^3}}S\left(\bar{p_1}M,n;r\ell_1\right)S\left(\bar{p_2}M,n;r\ell_2\right)\ll \frac{N^2P^3}{M^3t^3}\sqrt{r\ell_1\cdot r\ell_2}.
\end{equation*}
After the Poisson summation in the $n$-sum, we gain square-root cancellation for the character sum
\begin{equation*}
\sum_{a(r\ell_1\ell_2)}S\left(\bar{p_1}M,a;r\ell_1\right)S\left(\bar{p_2}M,a;r\ell_2\right)e\left(\frac{an}{r\ell_1\ell_2}\right)
\end{equation*}
in ``generic" cases, 
so that the dual $n$-sum is dominated by $r^{3/2}\ell_1\ell_2$. We save $\frac{N^{3/2}P^{5/2}}{M^{5/2}t^{5/2}L^{1/2}}$, which is more than $\frac{NP^3}{M^2t^2}$ if $N/Mt>PL$. Hence $\mathcal{F}$ is fine if $\frac{NP^2}{M^2t^2}<L<\frac{N}{PMt}$. 

Now it turns out that we can choose optimally $P=(Mt)^{5/18}$ and $L=(Mt)^{1/9}$ to simultaneously beat the bound $O(N^{1+\varepsilon})$ for $\mathcal{F}$ and $\mathcal{O}$, which in turn implies a nontrivial bound for $S(N)$, for $(Mt)^{3/2-1/18}<N<(Mt)^{3/2+\varepsilon}$. This yields a subconvexity bound 
$L\left(1/2+it,\pi\otimes \chi \right)\ll_{\pi} (M(|t|+1))^{3/4-1/36+\varepsilon}$.

\section{Some lemmas}

In this section, we collect some lemmas that we may use in our proof.

Let $(\lambda(m,n))_{m,n\neq 0}$ be the Fourier coefficients of the $\rm SL(3,\mathbb{Z})$ Hecke--Maass cusp form $\pi$.

First we have the following Rankin--Selberg estimate (see for example \cite{Molteni}).
\begin{lem}\label{Rankin--Selberg}Given any $\varepsilon>0$, one has
$$\sumsum_{m^2n\leq X}|\lambda(m,n)|^2\ll X^{1+\varepsilon}.$$
\end{lem}
From the lemma, we readily have the similar estimate
\begin{equation}\label{first-moment}
\sum_{n\leq X}|\lambda(1,n)|\ll X^{1+\varepsilon},
\end{equation}
by the Cauchy--Schwarz inequality.

Following \cite{Young17} and \cite{{Kiral-Petrow-Young}}, we make the following definition.
\begin{defn}\label{definition-of-inert}
We say a smooth function $f(x_1,...,x_n)$ on $\mathbb{R}^n$ is \emph{inert} if
\begin{equation}\label{inert}
x_1^{j_1} \cdots x_n^{j_n}f^{(j_1,...,j_n)}(x_1,...,x_n)\ll_{j_1,...,j_n} 1,
\end{equation}
for all nonnegative integers $j_1,...,j_n$.
Here the superscript denotes partial
differentiation.
\end{defn}

For any $N\geq 1$,
let 
\begin{equation}\label{initial sum}
S(N)= \sum_{n\geq1} \lambda(1,n)\chi(n) n^{-it}\varpi\left(\frac{n}{N}\right),
\end{equation}
where $\varpi(x)$ is an \emph{inert} function on $\mathbb{R}$ with compact support contained in $\mathbb{R}_{>0}$.

By symmetry, we assume $t>2$ from now on. Using a standard approximate functional equation argument (\cite[Theorem 5.3]{Iw-Ko}) and the estimate \eqref{first-moment}, one can derive the following.
\begin{lem}\label{approximate FE}
For any $\delta>0$ and $\varepsilon>0$, we have
\begin{equation*}
L\left(\frac{1}{2}+it,\pi\otimes \chi \right)\ll (Mt)^{\varepsilon}\sup_{N}\frac{|S(N)|}{\sqrt{N}}+(Mt)^{3/4-\delta/2+\varepsilon},
\end{equation*}
where the supremum is taken over $N$ in the range $(Mt)^{3/2-\delta}<N<(Mt)^{3/2+\varepsilon}$.
\end{lem}
From Lemma \ref{approximate FE}, it suffices to beat the convexity bound $N^{1+\varepsilon}$ for $S(N)$ for $N$ in the range $(Mt)^{3/2-\delta}<N<(Mt)^{3/2+\varepsilon}$, which we henceforth assume. Here $0<\delta<1/2$ is a small constant to be optimized later. 
 We observe for later convenience that
\begin{equation}\label{size of N}
(Mt)^{1+\varepsilon}<N.
\end{equation}

Let $\boldsymbol{\alpha}=(\alpha_1,\alpha_2,\alpha_3)$ be the Langlands parameters associated to the Maass cusp form $\pi$, with  
\begin{equation*}
\alpha_1+\alpha_2+\alpha_3=0, \quad \mbox{and}\quad |\Re \alpha_i|\leq \theta.
\end{equation*}
We recall the Ramanujan--Selberg conjecture predicts $\theta=0$, while from a result \cite{Jacquet-Shalika81} of Jacquet and Shalika one at least knows that $|\Re \alpha_i|< 1/2$.

Let
\begin{equation*}
G_{\delta}(s):= \begin{cases} 2(2\pi)^{-s}\Gamma(s)\cos(\pi s/2), & \quad  \text{if}\, \delta=0,\\
2i(2\pi)^{-s}\Gamma(s)\sin(\pi s/2), & \quad \text{if}\, \delta=1,\\
\end{cases}
\end{equation*}
and let
\begin{equation*}
G_{(\boldsymbol{\alpha},\boldsymbol{\delta})}(s)=\prod_{j=1}^{3}G_{\delta_j}(s+\alpha_j),
\end{equation*}
where $\boldsymbol{\delta}=(\delta_1,\delta_2,\delta_3)$.

Define 
\begin{equation*}
j_{(\boldsymbol{\alpha},\boldsymbol{\delta})}(x)=\frac{1}{2\pi i}\int_{\mathcal{C}}G_{(\boldsymbol{\alpha},\boldsymbol{\delta})}(s)x^{-s}\mathrm{d}s,\quad x>0,
\end{equation*}
where $\mathcal{C}$ is a curved contour such that all the singularities of $G_{(\boldsymbol{\alpha},\boldsymbol{\delta})}(s)$ are to the left of $\mathcal{C}$, defined as in Definition 3.2 of \cite{Qi2}.

Let 
\begin{equation*}
J_{\pi,\pm}(x):=J_{(\boldsymbol{\alpha},\boldsymbol{\delta})}(\pm x)=\frac{1}{2}\left(j_{(\boldsymbol{\alpha},\boldsymbol{\delta})}(x)\pm j_{(\boldsymbol{\alpha},\boldsymbol{\delta}+\boldsymbol{e})}(x)\right),
\end{equation*}
where $\boldsymbol{e}=(1,1,1)$, and $\boldsymbol{\delta}+\boldsymbol{e}$ is taken modulo $2$.
The Bessel function $J_{\pi,\pm}( x)$ satisfies the following properties.

%

\begin{lem}\label{J-Bessel}
(1). Let $\rho>\max\{-\Re \alpha_1, -\Re \alpha_2, -\Re \alpha_3 \}$. For $x\ll 1$, we have 
\begin{equation*}
x^jJ^{(j)}_{\pi,\pm}(x)\ll_{\alpha_1,\alpha_2,\alpha_3, \rho, j} x^{-\rho}.
\end{equation*}

(2). Let $K\geq 0$ be a fixed nonnegative integer. For $x>0$, we may write
\begin{equation*}\label{J_F}
J_{\pi,\pm}( x^3)=\frac{e(\pm 3x)}{x}W_\pi^{\pm}(x)+E_\pi^{\pm}(x),
\end{equation*}
where $W_\pi^{\pm}(x)$ and $E_\pi^{\pm}(x)$ are real-analytic functions on $(0,\infty)$ satisfying 
\begin{equation*}\label{W_F}
W_\pi^{\pm}(x)=\sum_{m=0}^{K-1}B^{\pm}_m(\pi) x^{-m}+O_{K,\alpha_1,\alpha_2,\alpha_3}\left(x^{-K}\right),
\end{equation*}
and
\begin{equation*}
E_\pi^{\pm,(j)}(x)\ll_{\alpha_1,\alpha_2,\alpha_3, j} \frac{\exp(-3\sqrt{3}\pi x)}{x},
\end{equation*}
for $x\gg_{\alpha_1,\alpha_2,\alpha_3} 1$, where $B^{\pm}_m(\pi)$ are constants depending on $\alpha_1$, $\alpha_2$ and $\alpha_3$.
\end{lem}
\begin{proof}
See \cite[Theorem 14.1]{Qi2}; note that our $J_{\pi,\pm}( x)$ is the $J_{(\lambda,\delta)}(x^{1/3})$ in the notation of \cite{Qi2}.
\end{proof}

Now we recall the Voronoi formula for $\rm GL(3)$, in which the Bessel function $J_{\pi,\pm}( x)$ appears naturally.
\begin{lem}[\cite{Miller-Schmid}]\label{voronoi} 
For $(a,c)=1$, $\bar{a}a\equiv 1(\bmod{\,c})$, we have
\begin{equation}\label{Voronoi00}
\sum_{n=1}^{\infty}\lambda(m,n)e\left(-\frac{na}{c}\right)w(n)=c\sum_{\pm}\sum_{m'|mc}\sum_{n=1}^{\infty}\frac{\lambda(n,m')}{m'n}S(\bar{a}m,\pm n;mc/m')\,\frac{m'^2n}{mc^3}W^{\pm}\left(\frac{m'^2n}{mc^3}\right),
\end{equation}
where
\begin{equation*}
W^{\pm}(x)=\int_0^\infty w(y)J_{\pi,\mp}(xy)\mathrm{d}y.
\end{equation*}
\end{lem}

In particular, replacing $w(n)$ by $w\left(\frac{n}{N}\right)$ gives
\begin{equation*}\label{scaled voronoi}
\begin{split}
&\sum_{n=1}^{\infty}\lambda(m,n)e\left(-\frac{na}{c}\right)w\left(\frac{n}{N}\right)\\
&\quad\quad\quad=c\sum_{\pm}\sum_{m'|mc}\sum_{n=1}^{\infty}\frac{\lambda(n,m')}{m'n}S(\bar{a}m,\pm n;mc/m')\,\frac{Nm'^2n}{mc^3}W^{\pm}\left(\frac{Nm'^2n}{mc^3}\right).
\end{split}
\end{equation*}
If $w^{(j)}(y)\ll 1$, then from the oscillation of $J_{\pi,\pm}(x)$ when $|x|>N^{\varepsilon}$, $W^{\pm}\left(\frac{Nm'^2n}{mc^3}\right)$ is negligibly small as long as $m'^2n$ is such that $\frac{Nm'^2n}{mc^3}\gg N^{\varepsilon}$. 

If we write
\begin{equation*}
\mathcal{U}^{\pm}(x)=xW^{\pm}(x),
\end{equation*}
then \eqref{Voronoi00} becomes
\begin{equation}\label{Voronoi0}
\sum_{n=1}^{\infty}\lambda(m,n)e\left(-\frac{na}{c}\right)w(n)=c\sum_{\pm}\sum_{m'|mc}\sum_{n=1}^{\infty}\frac{\lambda(n,m')}{m'n}S(\bar{a}m,\pm n;mc/m')\,\mathcal{U}^{\pm}\left(\frac{m'^2n}{mc^3}\right),
\end{equation}
which is the usual version of Voronoi formula given in the work \cite{Miller-Schmid} and others.



\begin{rem}
Here the normalization of \eqref{Voronoi00} is different from the usual version \eqref{Voronoi0}. With this normalization, the weight function on the right is the Hankel transform of the original Schwartz class function, matching the rank one and rank two cases. We thank Zhi Qi for making us aware of this. 
\end{rem}

\begin{lem}[Miller's bound, \cite{Mil1}]\label{Miller's}
Uniformly in $\alpha\in \mathbb{R}$, we have
\begin{equation}\label{Miller's bound}
\sum_{n\leq X}\lambda(1,n)e(\alpha n)\ll_{\pi,\varepsilon}X^{\frac{3}{4}+\varepsilon}.
\end{equation}
\end{lem}

\begin{lem}[{\cite[Lemma 2]{HN17}}]\label{ell_1 not ell_2}
Let $s_1$ and $s_2$ be natural numbers. Let $t_1,t_2$, and $n$ be integers. Set
\begin{equation*}
\mathcal{C}:=\sum_{x([s_1,s_2])}S(t_1x,1;s_1)S(t_2x,1;s_2)e\left(\frac{nx}{[s_1,s_2]}\right).
\end{equation*}
Write $s_i=w_i(s_1,s_2)$, $i=1,2$, and set $\Delta=w_2^2t_1-w_1^2t_2$. Then
\begin{equation*}
\begin{split}
|\mathcal{C}|\leq 2^{O(\omega([s_1,s_2]))}\left(s_1s_2[s_1,s_2]\right)^{1/2}\frac{(\Delta,n,s_1,s_2)}{(n,s_1,s_2)^{1/2}},
\end{split}
\end{equation*}
where $\omega([s_1,s_2])$ denotes the number of distinct prime factors of $[s_1,s_2]$, and the implied constant in $O$-symbol is absolute.
\end{lem}

\begin{lem}\label{Key lemma}
Let $V$ be a smooth function with compact support in $\mathbb{R}_{> 0}$, satisfying $V^{(j)}(x)\ll_j 1$ for all $j\geq 0$. Assume $(M,r)=1$ and $n\asymp N$. For any integer $A\geq 1$, there exists an inert function $V_A(x)$ compactly supported in $\mathbb{R}_{>0}$ such that
\begin{equation}\label{variant to HN}
\begin{split}
&\sum_{r=1}^{\infty}\chi(r)r^{-it}e\left(-\frac{n\bar{M}}{r}\right)V\left(\frac{r}{N/Mt}\right)\\
=&\frac{N}{M^{3/2}t^{3/2}}\frac{g_{\bar{\chi}}}{\sqrt{M}}\left(\frac{2\pi}{Mt}\right)^{-it}e(-t/2\pi)\chi(n)n^{-it}V_A\left(\frac{2\pi n}{N}\right)+O\left(\frac{N}{M^{3/2} t^{1+A}}\right)\\
&+\frac{1}{M}\left(\frac{N}{M t}\right)^{1-it}\sum_{\tilde{r}\neq 0}S_{\bar{\chi}}(n,\tilde{r};M)\int_{\mathbb{R}}x^{-it}e\left(-\frac{nt}{Nx}\right)V\left(x\right)e\left(-\frac{\tilde{r}N}{M^2 t}x\right)\mathrm{d}x,
\end{split}
\end{equation}
where $S_{\bar{\chi}}(n,\tilde{r};M)$ is the generalized Kloosterman sum.
\end{lem}
\begin{proof}
Writing
\begin{equation*}
e\left(-\frac{n\bar{M}}{r}\right)=e\left(\frac{n\bar{r}}{M}\right)e\left(-\frac{n}{Mr}\right),
\end{equation*}
which follows from reciprocity, and applying Poisson summation, the $r$-sum becomes
\begin{equation*}
\begin{split}
&\sum_{r=1}^{\infty}\chi(r)e\left(\frac{n\bar{r}}{M}\right)r^{-it}e\left(-\frac{n}{M r}\right)V\left(\frac{r}{N/Mt}\right)\\
=&\frac{N}{M^2 t}\sum_{\tilde{r}\in \mathbb{Z}}\sum_{a(M)}\chi(a)e\left(\frac{n\bar{a}}{M}\right)e\left(\frac{a\tilde{r}}{M}\right)\int_{\mathbb{R}}\left(\frac{N}{M t}x\right)^{-it}e\left(-\frac{nt}{Nx}\right)V\left(x\right)e\left(-\frac{\tilde{r}N}{M^2t}x\right)\mathrm{d}x.
\end{split}
\end{equation*}

In particular, the zero frequency $\tilde{r}=0$ contribution is
\begin{equation*}
\begin{split}
\frac{1}{M}\left(\frac{N}{Mt}\right)^{1-it}g_{\bar{\chi}}\chi\left(n\right)\int_{\mathbb{R}}x^{-it}e\left(-\frac{nt}{Nx}\right)V\left(x\right)\mathrm{d}x.
\end{split}
\end{equation*}

Considering the integral,
by \cite[Main Theorem]{Kiral-Petrow-Young}, there is an inert function $V_A$ supported on $x_0 \asymp 1$ such that
\begin{equation*}
\begin{split}
\int_{\mathbb{R}}x^{-it}e\left(-\frac{nt}{Nx}\right)V\left(x\right)\mathrm{d}x
=&\int_{\mathbb{R}}e\left(-\frac{t\log x}{2\pi}-\frac{nt}{Nx}\right)V\left(x\right)\mathrm{d}x\\
=&\frac{e(f(x_0))}{\sqrt{t}}V_A(x_0)+O_A\left(t^{-A}\right),
\end{split}
\end{equation*}
where $f(x)=-\frac{t\log x}{2\pi}-\frac{nt}{Nx}$, and $x_0=\frac{2\pi n}{N}$ is the unique solution for $f^\prime(x)=0$, and $A\geq 1$ is any arbitrarily large constant.
Therefore, 
\begin{equation*}
\begin{split}
&\int_{\mathbb{R}}x^{-it}e\left(-\frac{nt}{Nx}\right)V\left(x\right)\mathrm{d}x
=\left(\frac{2\pi}{N}\right)^{-it}\frac{e(-t/2\pi)}{\sqrt{t}}n^{-it}V_A\left(\frac{2\pi n}{N}\right)+O(t^{-A}).
\end{split}
\end{equation*}
Hence 
\begin{equation*}
\begin{split}
&\sum_{r=1}^{\infty}\chi(r)e\left(\frac{n\bar{r}}{M}\right)r^{-it}e\left(-\frac{n}{Mr}\right)V\left(\frac{r}{N/Mt}\right)\\
=&\frac{1}{M}\left(\frac{N}{Mt}\right)^{1-it}g_{\bar{\chi}}\,\chi\left(n\right)\left(\frac{2\pi}{N}\right)^{-it}\frac{e(-t/2\pi)}{\sqrt{t}}n^{-it}V_A\left(\frac{2\pi n}{N}\right)+O\left(\frac{N}{M^{3/2}t^{1+A}}\right)\\
&+\frac{1}{M}\left(\frac{N}{Mt}\right)^{1-it}\sum_{\tilde{r}\neq 0}S_{\bar{\chi}}(n,\tilde{r};M)\int_{\mathbb{R}}x^{-it}e\left(-\frac{nt}{Nx}\right)V\left(x\right)e\left(-\frac{\tilde{r}N}{M^2t}x\right)\mathrm{d}x,
\end{split}
\end{equation*}
and \eqref{variant to HN} follows.
\end{proof}

\begin{rem}
The identity \eqref{variant to HN},
\begin{equation*}
\begin{split}
\chi(n)n^{-it}V_A\left(\frac{n}{N}\right)
=&\bigg(\frac{2\pi}{Mt}\bigg)^{it} e\bigg(\frac{t}{2\pi}\bigg)\frac{M^{2}t^{3/2}}{Ng_{\bar{\chi}}}\sum_{r=1}^{\infty}\chi(r)e\bigg(\frac{n\bar{r}}{M}\bigg)r^{-it}e\bigg(-\frac{n}{Mr}\bigg)V\left(\frac{r}{N/Mt}\right)\\
&-\bigg(\frac{2\pi}{N}\bigg)^{it}e\bigg(\frac{t}{2\pi}\bigg) \frac{t^{1/2}}{g_{\bar{\chi}}}\sum_{\tilde{r}\neq 0}S_{\bar{\chi}}(n,\tilde{r};M)\int_{\mathbb{R}}V(x)x^{-it}e\bigg(-\frac{nt}{Nx}\bigg)e\bigg(-\frac{\tilde{r}Nx}{M^2t}\bigg)\,\mathrm{d}x\\
&\quad+O\left(t^{1/2-A}\right),
\end{split}
\end{equation*}
is a variant of the following key identity in \cite[(3.6)]{HN17}.
\begin{equation*}
\begin{split}
\chi(n)\hat{V}(0)=\frac{M}{Rg_{\bar{\chi}}}\sum_{r\in\mathbb{Z}}\chi(r)e\left(\frac{n\bar{r}}{M}\right)V\left(\frac{r}{R}\right)-\frac{1}{g_{\bar{\chi}}}\sum_{\tilde{r}\neq 0}S_{\bar{\chi}}(n,\tilde{r};M)\hat{V}\left(\frac{\tilde{r}}{M/R}\right),
\end{split}
\end{equation*}where $R>0$ is a parameter and $\hat{V}$ denotes the Fourier transform of the Schwartz function $V$, which is compactly supported in $\mathbb{R}_{>0}$. Inserting the identity, with an amplification technique, one can express the smoothed sum $\sum_{n\geq1} \lambda(1,n) \chi(n)w\left(\frac{n}{N}\right)$ as $\mathcal{F}+\mathcal{O}$. Balancing the contribution of $\mathcal{F}$ and $\mathcal{O}$ properly, the authors of \cite{HN17} obtained $L\left(1/2,\pi\otimes \chi\right)\ll M^{3/4-1/36+\varepsilon}$.
\end{rem}

From Lemma \ref{Key lemma}, assuming $(M,\ell r)=1$ and $n\asymp N$, one has
\begin{equation}\label{variant to HN 2}
\begin{split}
&\sum_{r=1}^{\infty}\chi(r)r^{-it}e\left(-\frac{np\bar{M}}{\ell r}\right)V\left(\frac{r}{Np/M\ell t}\right)\\
=&\frac{Np}{M^{3/2}t^{3/2}\ell}\frac{g_{\bar{\chi}}}{\sqrt{M}}\left(\frac{2\pi p}{M\ell t}\right)^{-it}e(-t/2\pi)\chi\left(p\bar{\ell}\right)\chi(n)n^{-it}V_A\left(\frac{2\pi n}{N}\right)\\
&+\frac{1}{M}\left(\frac{Np}{M\ell t}\right)^{1-it}\chi(p\bar{\ell})\sum_{\tilde{r}\neq 0}S_{\bar{\chi}}(n,\tilde{r}p\bar{\ell};M)\mathcal{J}_{it}(n,\tilde{r}p/\ell;M)+O\left(\frac{Np}{M^{3/2}\ell t^{1+A}}\right),
\end{split}
\end{equation}
or, taking another form,
\begin{equation}\label{lin analogue2}
\begin{split}
\chi(n)n^{-it}&V_A\left(\frac{2\pi n}{N}\right)\\
=&\bigg(\frac{2\pi}{Mt}\bigg)^{it} e\bigg(\frac{t}{2\pi}\bigg)\frac{M^{2}t^{3/2}\ell}{Np g_{\bar{\chi}}}\sum_{r=1}^{\infty}\chi(r\ell \bar{p})\left(\frac{r\ell}{p}\right)^{-it}e\bigg(-\frac{np\bar{M}}{\ell r}\bigg)V\left(\frac{r}{Np/M\ell t}\right)\\
&-\bigg(\frac{2\pi}{N}\bigg)^{it}e\bigg(\frac{t}{2\pi}\bigg) \frac{t^{1/2}}{g_{\bar{\chi}}}\sum_{\tilde{r}\neq 0}S_{\bar{\chi}}(n,\tilde{r}p\bar{\ell};M)\mathcal{J}_{it}(n,\tilde{r}p/\ell;M)+O\left(t^{1/2-A}\right),
\end{split}
\end{equation}
where
\begin{equation}\label{J-integral}
\mathcal{J}_{it}(n,\tilde{r}p/\ell;M):=\int_{\mathbb{R}}x^{-it}e\left(-\frac{nt}{Nx}\right)V\left(x\right)e\left(-\frac{\tilde{r}Np}{M^2\ell t}x\right)\mathrm{d}x.
\end{equation}

We shall use \eqref{variant to HN 2} as a ``key identity" in our proof; see Section \ref{Reducing}.

\begin{lem}\label{spacing}
For any $\varepsilon>0$, one has
\begin{equation*}
\begin{split}
&\mathop{\sum_{p_1\sim P}\sum_{p_2\sim P}\sum_{\ell_1\sim L}\sum_{\ell_2\sim L}\sum_{r_1\sim R}\sum_{r_2\sim R}}_{\ell_1 r_2p_2\neq \ell_2 r_1p_1}\frac{1}{\left|\ell_1 r_2p_2-\ell_2 r_1p_1\right|}\ll (LPR)^{1+\varepsilon},
\end{split}
\end{equation*}
and
\begin{equation*}\mathop{\sum_{p_1\sim P}\sum_{p_2\sim P}\sum_{\ell_1\sim L}\sum_{\ell_2\sim L}\sum_{r_1\sim R}\sum_{r_2\sim R}}_{\substack{\ell_1 r_2p_2\neq \ell_2 r_1p_1\\ \ell_1 r_2p_2\equiv \ell_2 r_1p_1(M)}}\frac{1}{\left|r_1p_1\ell_2-r_2p_2\ell_1\right|}\ll \frac{(LPR)^{1+\varepsilon}}{M}.
\end{equation*}
\end{lem}
\begin{proof}
The first sum is bounded by
\begin{equation*}
\begin{split}
&\sum_{m_1\sim LPR}\mathop{\sum_{m_2\sim LPR}}_{m_2\neq  m_1}\frac{\tau_3(m_1)\tau_3(m_2)}{\left|m_1- m_2\right|}\leq (LPR)^{\varepsilon}\sum_{m\sim LPR}\sum_{1\leq h\ll LPR}\frac{1}{h}\ll (LPR)^{1+\varepsilon}.
\end{split}
\end{equation*}
Hence the first inequality follows. Here $\tau_3(m):=\sum_{abc=m}1$.

The second inequality can be proven using similar argument.


\end{proof}


\section{Reducing $S(N)$ to $\mathcal{F}_1$ and $\mathcal{O}$}\label{Reducing}
Our basic strategy is to introduce more ``points" of summation to mimic the smoothed sum $S(N)$ in \eqref{initial sum}, which is our main object of study. Through out the paper we assume that $|t|>M^{\varepsilon}$ for any $\varepsilon>0$.

Let $P$ and $L$ be two large parameters.
We begin by introducing the following sum
\begin{equation}\label{new F1 sum}
\begin{split}
\mathcal{F}_1=& \frac{M^{3/2}t^{3/2}}{NP^2}\sum_{p\sim P}\bar{\chi}(p)p^{it}\sum_{\ell\sim L}\chi(\ell)\ell^{-it}\sum_{r=1}^{\infty}\chi(r)r^{-it}V\left(\frac{r}{Np/M\ell t}\right)\sum_{n=1}^{\infty}\lambda(1,n)e\left(-\frac{np\bar{M}}{\ell r}\right)w\left(\frac{n}{N}\right),
\end{split}
\end{equation}
where $p\sim P$ and $\ell\sim L$ denote primes in the dyadic segments $[P,2P]$ and $[L,2L]$, respectively; $w$ and $V$ are smooth functions with compact supports in $\mathbb{R}_{> 0}$ satisfying $w^{(j)}(x), V^{(j)}(x)\ll_j 1$ for all $j\geq 0$. 

We shall see that if one applies Poisson summation to the $r$-sum (which is the content of Lemma \ref{Key lemma}), then the contribution of the zero frequency $\tilde{r}=0$ ($\tilde{r}$ the variable dual to $r$) will give rise to the sum $S(N)$ that we are initially interested in. In order to bound $S(N)$, it suffices to bound $\mathcal{F}_1$ and the sum arising from the non-zero frequencies $\tilde{r}\neq 0$ of the dual sum, which we denote by $\mathcal{O}$. This observation is initially due to Holowinsky and Nelson \cite[B.4]{HN17}, in their work in the Dirichlet character twist case.

Plugging the identity \eqref{variant to HN 2} in, we get
\begin{equation*}
\begin{split}
\mathcal{F}_1=&\left(\frac{2\pi}{Mt}\right)^{-it}e\left(-\frac{t}{2\pi}\right)\frac{g_{\bar{\chi}}}{M^{1/2}}\sum_{p\sim P}p/P^2\sum_{\ell\sim L}\ell^{-1}\sum_{n=1}^{\infty}\lambda(1,n)\chi(n)n^{-it}w\left(\frac{n}{N}\right)V_A\left(\frac{2\pi n}{N}\right)\\
&+\left(\frac{N}{Mt}\right)^{-it}\frac{t^{1/2}}{P^2M^{1/2}}\sum_{n=1}^{\infty}\lambda(1,n)w\left(\frac{n}{N}\right)\sum_{p\sim P}p\sum_{\ell\sim L}\ell^{-1}\\
&\quad\quad\sum_{\tilde{r}\neq 0}S_{\bar{\chi}}(n,\tilde{r}p\bar{\ell};M)\mathcal{J}_{it}(n,\tilde{r}p/\ell;M)+O(Nt^{1/2-A}),\end{split}
\end{equation*}
which implies 
\begin{equation*}
\begin{split}
 \frac{1}{\log P\log L}&\sum_{n=1}^{\infty}\lambda(1,n)\chi(n)n^{-it}w\left(\frac{n}{N}\right)V_A\left(\frac{2\pi n}{N}\right)\asymp |\mathcal{F}_1| +O(Nt^{1/2-A})\\+&\frac{t^{1/2}}{M^{1/2}PL}\bigg|\sum_{n=1}^{\infty}\lambda(1,n)w\left(\frac{n}{N}\right)\sum_{p\sim P}\sum_{\ell\sim L}\sum_{\tilde{r}\neq 0}S_{\bar{\chi}}(n,\tilde{r}p\bar{\ell};M)\mathcal{J}_{it}(n,\tilde{r}p/\ell;M)\bigg|.
\end{split}
\end{equation*}

We have shown the following.
\begin{lem}
For any positive integers $A\geq 1$, there exists an inert function $V_A(x)$ with compact support in $\mathbb{R}_{>0}$ such that asymptotically, one has
\begin{equation}\label{connection1}
\begin{split}
\frac{1}{\log P\log L}\sum_{n=1}^{\infty}\lambda(1,n)\chi(n)n^{-it}w\left(\frac{n}{N}\right)V_A\left(\frac{2\pi n}{N}\right)
\asymp |\mathcal{F}_1|+|\mathcal{O}|+O\left(Nt^{1/2-A}\right),
\end{split}
\end{equation}
with
\begin{equation*}
\begin{split}
\mathcal{F}_1=&\frac{M^{3/2}t^{3/2}}{NP^2}\sum_{p\sim P}\bar{\chi}(p)p^{it}\sum_{\ell\sim L}\chi(\ell)\ell^{-it}\sum_{r=1}^{\infty}\chi(r)r^{-it}V\left(\frac{r}{Np/M\ell t}\right)\sum_{n=1}^{\infty}\lambda(1,n)e\left(-\frac{np\bar{M}}{\ell r}\right)w\left(\frac{n}{N}\right),
\end{split}
\end{equation*}
and
\begin{equation}\label{the O1 sum}
\begin{split}
\mathcal{O}=&\frac{t^{1/2}}{M^{1/2}PL}\sum_{n=1}^{\infty}\lambda(1,n)w\left(\frac{n}{N}\right)\sum_{p\sim P}\sum_{\ell\sim L}\sum_{\tilde{r}\neq 0}S_{\bar{\chi}}(n,\tilde{r}p\bar{\ell};M)\mathcal{J}_{it}(n,\tilde{r}p/\ell;M),
\end{split}
\end{equation}
where $\mathcal{J}_{it}(n,\tilde{r}p/\ell;M)$ is given by \eqref{J-integral}.
\end{lem}

For any given $\varepsilon>0$, we can make the error term $O\left(Nt^{1/2-A}\right)$ negligibly small by assuming $|t|>M^{\varepsilon}$ and taking $A$ to be sufficiently large. It is easily seen that the function $\varpi(x):=w(x)V_A(2\pi x)$ is an \emph{inert} function (under Definition \ref{definition-of-inert}); see for instance \cite[Example 4]{Kiral-Petrow-Young}. From the lemma, to bound 
\begin{equation*}
\sum_{n=1}^{\infty}\lambda(1,n)\chi(n)n^{-it}w\left(\frac{n}{N}\right)V_A\left(\frac{2\pi n}{N}\right)=\sum_{n=1}^{\infty}\lambda(1,n)\chi(n)n^{-it}\varpi\left(\frac{n}{N}\right),
\end{equation*}
which is our original object of study \eqref{initial sum}, it suffices to bound the terms $\mathcal{F}_1$ and $\mathcal{O}$. Note that a priori $w(x)V_A(2\pi x)$ needs not be an arbitrarily bump function, but this can be done by adjusting the weight function $w(x)$ in our initial definition of $\mathcal{F}_1$ in \eqref{new F1 sum} appropriately (e.g. replacing the original $w(x)$ by $w(x)/V_A(2\pi x)$).

\section{Treatment of $\mathcal{O}$}\label{the treatment of O}
This section is devoted to giving a nontrivial bound for the sum
\begin{equation*}
\begin{split}
\mathcal{O}=&\frac{t^{1/2}}{M^{1/2}PL}\sum_{n=1}^{\infty}\lambda(1,n)w\left(\frac{n}{N}\right)\sum_{p\sim P}\sum_{\ell\sim L}\sum_{r\neq 0}S_{\bar{\chi}}(n,rp\bar{\ell};M)\mathcal{J}_{it}(n,rp/\ell;M),
\end{split}
\end{equation*}
introduced in \eqref{the O1 sum}. Here 
\begin{equation*}
\mathcal{J}_{it}(n,rp/\ell;M)=\int_{\mathbb{R}}V(x)x^{-it}e\left(-\frac{nt}{Nx}\right)e\left(-\frac{rNpx}{M^2\ell t}\right)\mathrm{d}x,
\end{equation*}defined in \eqref{J-integral}.

Our goal is to improve the bound $\mathcal{O}=O\left(N^{1+\varepsilon}\right)$.

For $r\neq 0$, integrating by parts implies that the integral $\mathcal{J}_{it}(n,rp/\ell;M)$
is negligibly small, unless $0\neq |r|\leq N^{\varepsilon} \frac{M^2t^2L}{NP}$ (by \cite[Lemma 8.1]{BKY}). Moreover, using the second derivative test (\cite[Lemma 5.1.3]{Huxley}) we find that $\mathcal{J}_{it}(n,rp/\ell;M)\ll t^{-1/2+\varepsilon}$. 

To estimate $\mathcal{O}$, by a dyadic subdivision, it suffices to bound the sum
\begin{equation*}
\begin{split}
\mathcal{O}(R):=&\frac{t^{1/2}}{M^{1/2}PL}\sum_{n=1}^{\infty}\lambda(1,n)w\left(\frac{n}{N}\right)\sum_{p\sim P}\sum_{\ell\sim L}\sum_{r\sim R}S_{\bar{\chi}}(n,rp\bar{\ell};M)\mathcal{J}_{it}(n,rp/\ell;M),
\end{split}
\end{equation*}
where $R$ satisfies
\begin{equation*}
\begin{split}
1\ll R \ll N^{\varepsilon} \frac{M^2t^2L}{NP}.
\end{split}
\end{equation*}

By the Cauchy--Schwarz inequality and Lemma \ref{Rankin--Selberg}, 
\begin{equation}\label{O1(R)}
\begin{split}
\mathcal{O}(R)\ll& \frac{N^{1/2+\varepsilon}t^{1/2}}{M^{1/2}PL}\left(\sum_{n=1}^{\infty}\left|\sum_{p\sim P}\sum_{\ell\sim L}
\sum_{r\sim R}S_{\bar{\chi}}(n,rp\bar{\ell};M)\mathcal{J}_{it}(n,rp/\ell;M)\right|^2w\left(\frac{n}{N}\right)\right)^{1/2}\\
=&\frac{N^{1/2+\varepsilon}t^{1/2}}{M^{1/2}PL}\bigg(\sum_{p_1\sim P}\sum_{p_2\sim P}\sum_{\ell_1\sim L}\sum_{\ell_2\sim L}\sum_{r_1 \sim R}\sum_{r_2\sim R}\\
&\sum_{n=1}^{\infty}S_{\bar{\chi}}(n,r_1p_1\bar{\ell_1};M)\overline{S_{\bar{\chi}}(n,r_2p_2\bar{\ell_2};M)}\mathcal{J}_{it}(n,r_1p_1/\ell_1;M)\overline{\mathcal{J}_{it}(n,r_2p_2/\ell_2;M)}\,w\left(\frac{n}{N}\right)\bigg)^{1/2}.
\end{split}
\end{equation}

Next, we apply Poisson summation to the $n$-sum, yielding
\begin{equation*}
\begin{split}
&\sum_{n=1}^{\infty}S_{\bar{\chi}}(n,r_1p_1\bar{\ell_1};M)\overline{S_{\bar{\chi}}(n,r_2p_2\bar{\ell_2};M)}\mathcal{J}_{it}(n,r_1p_1/\ell_1;M)\overline{\mathcal{J}_{it}(n,r_2p_2/\ell_2;M)}\,w\left(\frac{n}{N}\right)\\
=&\frac{N}{M}\sum_{n\in\mathbb{Z}}\sum_{a(M)}S_{\bar{\chi}}(a,r_1p_1\bar{\ell_1};M)\overline{S_{\bar{\chi}}(a,r_2p_2\bar{\ell_2};M)}\,e\left(\frac{an}{M}\right)\\
&\quad\quad\int_{\mathbb{R}}\mathcal{J}_{it}(Ny,r_1p_1/\ell_1;M)\overline{\mathcal{J}_{it}(Ny,r_2p_2/\ell_2;M)}\,w\left(y\right)e\left(-\frac{nN}{M}y\right)\mathrm{d}y.
\end{split}
\end{equation*}
Taking into account the oscillations of $\mathcal{J}_{it}(Ny,r_1p_1/\ell_1;M)$ and $\overline{\mathcal{J}_{it}(Ny,r_2p_2/\ell_2;M)}$, the integral is arbitrarily small for $n\neq 0$ (since $N\gg (Mt)^{1+\varepsilon}$). Hence only the zero frequency contributes significantly for the dual sum:
\begin{equation}\label{n sum after Poisson1}
\begin{split}
&\sum_{n=1}^{\infty}S_{\bar{\chi}}(n,r_1p_1\bar{\ell_1};M)\overline{S_{\bar{\chi}}(n,r_2p_2\bar{\ell_2};M)}\mathcal{J}_{it}(n,r_1p_1/\ell_1;M)\overline{\mathcal{J}_{it}(n,r_2p_2/\ell_2;M)}\,w\left(\frac{n}{N}\right)\\
&=\frac{N}{M}\mathfrak{C}\,\mathfrak{J}+O(N^{-2018}),
\end{split}
\end{equation}
where
\begin{equation*}
\begin{split}
\mathfrak{C}=\sum_{a(M)}S_{\bar{\chi}}(a,r_1p_1\bar{\ell_1};M)\overline{S_{\bar{\chi}}(a,r_2p_2\bar{\ell_2};M)}
=&M\sumstar_{\beta(M)}e\left(\frac{(r_1p_1\bar{\ell_1}-r_2p_2\bar{\ell_2})\beta}{M}\right)\\
=&M\left[M\delta_{\ell_2r_1p_1\equiv \ell_1r_2p_2(M)}-1\right]
\end{split}
\end{equation*}
and
\begin{equation}\label{the integral J1}
\begin{split}
\mathfrak{J}=\int_{\mathbb{R}}\mathcal{J}_{it}(Ny,r_1p_1/\ell_1;M)\overline{\mathcal{J}_{it}(Ny,r_2p_2/\ell_2;M)}\,w\left(y\right)\mathrm{d}y.
\end{split}
\end{equation}

One readily sees that 
\begin{equation}\label{C1 character sum}
\mathfrak{C}= \begin{cases} O(M^2), & \quad  \ell_1 r_2p_2\equiv \ell_2 r_1p_1(M)\\
O(M), & \quad \mathrm{otherwise}.\\
\end{cases}
\end{equation}

For the integral $\mathfrak{J}$, if we use the previously mentioned second derivative bound $\mathcal{J}_{it}(n,rp/\ell;M)\ll t^{-1/2+\varepsilon}$ we get $\mathfrak{J}\ll t^{-1+\varepsilon}$. However, there are more cancellations beyond $O(t^{-1+\varepsilon})$, as long as the parameters $(r_i,p_i,\ell_i)$ satisfy $r_1p_1\ell_2\neq r_2p_2\ell_1$. Indeed, we have the following precise estimate in terms of $(r_i,p_i,\ell_i)$.
\begin{lem}\label{bound of J}
For $\mathfrak{J}$ defined as in \eqref{the integral J1}, we have
\begin{equation*}
\begin{split}
\mathfrak{J}\ll t^{\varepsilon}(\max\left\{t,|X|\right\})^{-1},
\end{split}
\end{equation*}
where $X:=\frac{N(  \ell_2 r_1p_1-\ell_1 r_2p_2)}{M^2\ell_1\ell_2}$.
\end{lem}
This can be proven by using stationary phase expansion for each of the $\mathcal{J}_{it}(Ny,rp/\ell;M)$ in the definition of $\mathfrak{J}$ and then applying the first derivative test. The following simpler proof is suggested to us by the referee.
\begin{proof}
By definition \eqref{J-integral} of $\mathcal{J}_{it}(Ny,rp/\ell;M)$, we express $\mathfrak{J}$ as  $\mathfrak{J}=\int V(u)V(v)w(y)e(f(u)-f(v))\mathrm{d}y\mathrm{d}v\mathrm{d}u$. Observe that
$$f(u)-f(v)=\frac{t}{uv}\left(u-v\right)y+\left(-\frac{t\log u}{2\pi}+\frac{t\log v}{2\pi}+\frac{Nr_2p_2}{M^2t\ell_2}v-\frac{Nr_1p_1}{M^2t\ell_1}u\right)$$
and the integral $\int w(y)e(ty(u-v)/uv)\mathrm{d}y$ is negligible if $|u-v|\gg t^{\varepsilon-1}$. Let $q$ be a non-negative smooth function supported on $[-1,1]$ which takes the value $1$ on $[-1/2,1/2]$. For arbitrarily large constant $A$, set 
$$Q(u-v):=q\left(\frac{u-v}{t^{\varepsilon-1}}\right)\left(1-q\left(\frac{u-v}{t^{-A}}\right)\right),$$
supported in $t^{-A}\ll |u-v|\ll t^{\varepsilon-1}$. Then $$\mathfrak{J}=\int Q(u-v)V(u)V(v)w(y)e(f(u)-f(v))\mathrm{d}y\mathrm{d}v\mathrm{d}u+O_{A,\varepsilon}(t^{-A}).$$
Integrating by parts, the integral is
\begin{equation}\label{expansion of J}
    \frac{1}{2\pi i\cdot t}\int Q(u-v)\frac{uV(u)vV(v)}{u-v}w^\prime(y)e(f(u)-f(v))\mathrm{d}y\mathrm{d}v\mathrm{d}u.
\end{equation}
If $|X|\ll t^{1+2\varepsilon}$, then we are done for the integral in \eqref{expansion of J} is trivially $
\ll \log t$. Otherwise, a change of variable $h=u-v$ leads the integral to
$$\int \frac{Q(h)}{h}w^\prime(y)(v+h)vV(v+h)V(v)e(f(v+h)-f(v))\mathrm{d}v\mathrm{d}h\mathrm{d}y.$$
Now 
$$f(v+h)-f(v)=\frac{X}{t}v-\frac{t}{2\pi}\log (1+\frac{h}{v})+th\frac{y}{v(v+h)}-\frac{Nr_1p_1}{M^2t\ell_1}h,$$
hence
$$\left|\frac{d}{dv}(f(v+h)-f(v))\right|\gg \frac{|X|}{t}-O(th)\gg \frac{|X|}{t}-O(t^{\varepsilon})\gg \frac{|X|}{t}.$$
It follows by the first derivative test that the expression in \eqref{expansion of J} $\ll |X|^{-1}\log t$.
\end{proof}

\begin{rem}
For $\ell_1 r_2p_2\neq \ell_2 r_1p_1$ and $r_i\sim \frac{M^2t^2L}{NP}$, typically $|X|^{-1}\asymp t^{-2}$,
so that the second bound of the lemma shows that we save an extra $t$ over the `trivial bound' $O(t^{-1+\varepsilon})$.
The estimation of this lemma is an analytic analogue of the bound \eqref{C1 character sum}. 
\end{rem}

Now we return to the estimate of $\mathcal{O}(R)$, in \eqref{O1(R)}.
Plugging the $n$-sum \eqref{n sum after Poisson1} into $\mathcal{O}(R)$, up to a negligible error, we have
\begin{equation}\label{counting1}
\begin{split}
\mathcal{O}(R)\ll&\frac{N^{1/2+\varepsilon}t^{1/2}}{M^{1/2}PL}\left(\sum_{p_1\sim P}\sum_{p_2\sim P}\sum_{\ell_1\sim L}\sum_{\ell_2\sim L}\sum_{r_1\sim R}\sum_{r_2\sim R}\frac{N}{M}|\mathfrak{C}|\,|\mathfrak{J}|\right)^{1/2}\\
\ll&\frac{N^{1+\varepsilon}t^{1/2}}{MPL}\bigg(\mathop{\sum_{p_1\sim P}\sum_{p_2\sim P}\sum_{\ell_1\sim L}\sum_{\ell_2\sim L}\sum_{r_1\sim R}\sum_{r_2\sim R}}_{\ell_1 r_2p_2\equiv \ell_2 r_1p_1(M)}M^2|\mathfrak{J}|\\
&\quad\quad+\mathop{\sum_{p_1\sim P}\sum_{p_2\sim P}\sum_{\ell_1\sim L}\sum_{\ell_2\sim L}\sum_{r_1\sim R}\sum_{r_2\sim R}}_{\ell_1 r_2p_2\not \equiv \ell_2 r_1p_1(M)}M\frac{M^2\ell_1\ell_2}{N\left| \ell_1 r_2p_2- \ell_2 r_1p_1\right|}\bigg)^{1/2},
\end{split}
\end{equation}
by using \eqref{C1 character sum} and Lemma \ref{bound of J}. We remind the reader that $R$ satisfies $1\leq R \ll N^{\varepsilon} \frac{M^2t^2L}{NP}$.

Using Lemma \ref{bound of J} again, the first term inside the parentheses is bounded by 
\begin{equation*}
\begin{split}
& M^2t^{\varepsilon}\mathop{\sum_{p_1\sim P}\sum_{p_2\sim P}\sum_{\ell_1\sim L}\sum_{\ell_2\sim L}\sum_{r_1\sim R}\sum_{r_2\sim R}}_{\ell_1 r_2p_2= \ell_2 r_1p_1}t^{-1}
+\frac{M^4L^2t^{\varepsilon}}{N}\mathop{\sum_{p_1\sim P}\sum_{p_2\sim P}\sum_{\ell_1\sim L}\sum_{\ell_2\sim L}\sum_{r_1\sim R}\sum_{r_2\sim R}}_{\substack{\ell_1 r_2p_2\neq \ell_2 r_1p_1\\ \ell_1 r_2p_2\equiv \ell_2 r_1p_1(M)}}\frac{1}{\left| \ell_1 r_2p_2- \ell_2 r_1p_1\right|},
\end{split}
\end{equation*}
which is further dominated by
\begin{equation*}
\begin{split}
&\ll N^{\varepsilon}M^2t^{-1}PLR+N^{\varepsilon}\frac{M^4L^2}{N}\cdot\frac{LPR}{M}\\
&\ll N^{\varepsilon}\frac{M^4tL^2}{N}+N^{\varepsilon}\frac{M^5t^2L^4}{N^2},
\end{split}
\end{equation*}
by using Lemma \ref{spacing} and by noting that $R \ll N^{\varepsilon} \frac{M^2t^2L}{NP}$.


Similarly, the second term inside the parentheses of \eqref{counting1} is bounded by
\begin{equation*}
\begin{split}
\frac{M^3L^2}{N}\mathop{\sum_{p_1\sim P}\sum_{p_2\sim P}\sum_{\ell_1\sim L}\sum_{\ell_2\sim L}\sum_{r_1\sim R}\sum_{r_2\sim R}}_{\ell_1 r_2p_2\neq \ell_2 r_1p_1}\frac{1}{\left|\ell_1 r_2p_2-\ell_2 r_1p_1\right|}
&\ll N^{\varepsilon}\frac{M^3L^2}{N}\cdot PLR\ll N^{\varepsilon}\frac{M^5t^2L^4}{N^2},
\end{split}
\end{equation*}
upon using Lemma \ref{spacing}.

Returning to the estimate of $\mathcal{O}(R)$, we have shown that
\begin{equation*}
\begin{split}
\mathcal{O}(R)\ll&\frac{N^{1+\varepsilon}t^{1/2}}{MPL}\left(\frac{M^4tL^2}{N}+\frac{M^5t^2L^4}{N^2}\right)^{1/2}
\ll \frac{N^{1/2+\varepsilon}Mt}{P}+N^{\varepsilon}\frac{M^{3/2}t^{3/2}L}{P},
\end{split}
\end{equation*}
which holds for any $1\leq R \ll N^{\varepsilon} \frac{M^2t^2L}{NP}$,


We summarize the main result of this section.
\begin{prop}\label{final bound for O}
 For any $\varepsilon>0$, we have the bound
\begin{equation*}
\mathcal{O}\ll \frac{N^{1/2+\varepsilon}Mt}{P}+N^{\varepsilon}\frac{M^{3/2}t^{3/2}L}{P},
\end{equation*}
for $\mathcal{O}$ defined as in \eqref{the O1 sum}.
\end{prop}
\begin{rem}
If we only use the `trivial' bound $\mathfrak{J}\ll t^{-1+\varepsilon}$ for the estimate of the integral $\mathfrak{J}$, then one will see that for the second term we get $O\left(N^{\varepsilon}M^{3/2}t^{2}L/P\right)$ instead. It is thus crucial to use Lemma \ref{bound of J} to get an extra $t^{1/2}$ saving in order to beat the convexity bound for $L(1/2+it,\pi\otimes \chi)$ in the $t$-aspect.
\end{rem}

\section{Treatment of $\mathcal{F}_1$}
The purpose of this section is to give a nontrivial bound for 
\begin{equation*}
\begin{split}
\mathcal{F}_1=&\frac{M^{3/2}t^{3/2}}{NP^2}\sum_{p\sim P}\bar{\chi}(p)p^{it}\sum_{\ell\sim L}\chi(\ell)\ell^{-it}\sum_{r=1}^{\infty}\chi(r)r^{-it}V\left(\frac{r}{Np/M\ell t}\right)
\sum_{n=1}^{\infty}\lambda(1,n)e\left(-\frac{np\bar{M}}{\ell r}\right)w\left(\frac{n}{N}\right),
\end{split}
\end{equation*}
defined in \eqref{new F1 sum}, where $w$ and $V$ are smooth compactly supported functions with bounded derivatives.

Our goal is to improve the bound $\mathcal{F}_1=O\left(N^{1+\varepsilon}\right)$.

Bounding the sum directly with Miller's bound \eqref{Miller's bound}, we have $\mathcal{F}_1\ll N^{3/4+\varepsilon}(Mt)^{1/2}$, which is not satisfactory yet for our purpose. 

We shall apply the Voronoi summation to the $n$-sum. To this end, one may assume $(p,r)=1$ in $\mathcal{F}_1$, as the contribution from the terms $(p,r)>1$ is negligible, compared to the generic terms $(p,r)=1$. We briefly justify this. Denote the terms with $p|r$ in $\mathcal{F}_1$ by $\mathcal{F}_1^{\sharp}$. Then,
\begin{equation*}
\begin{split}
\mathcal{F}_1^{\sharp}=\frac{M^{3/2}t^{3/2}}{NP^2}\sum_{p\sim P}\sum_{\ell\sim L}\chi(\ell)\ell^{-it}\sum_{r=1}^{\infty}\chi(r)r^{-it}V\left(\frac{r}{N/M\ell t}\right)\sum_{n=1}^{\infty}\lambda(1,n)e\left(-\frac{n\bar{M}}{\ell r}\right)w\left(\frac{n}{N}\right).
\end{split}
\end{equation*}
An application of Voronoi summation \eqref{Voronoi0} takes the $n$-sum to the following dual sum
\begin{equation*}
\begin{split}
\ell r\sum_{\pm}\sum_{m|\ell r}\sum_{n=1}^{\infty}\frac{\lambda(n,m)}{mn}S(M,\pm n;\ell r/m)\,\mathcal{U}^{\pm}\left(\frac{m^2n}{(\ell r)^{3}/N}\right),
\end{split}
\end{equation*}
where the new length can be truncated at $m^2n<N^{\varepsilon}(\ell r)^3/N$, at the cost of a negligible error.

Hence we can estimate $\mathcal{F}_1^{\sharp}$ as follows.
\begin{equation*}
\begin{split}
\mathcal{F}_1^{\sharp}&\asymp \frac{M^{3/2}t^{3/2}}{NP\log P}\bigg|\sum_{\ell\sim L}\chi(\ell)\ell^{-it}\sum_{r=1}^{\infty}\chi(r)r^{-it}V\left(\frac{r}{N/M\ell t}\right)\\
&\quad\quad \ell r\sum_{m|\ell r}\sum_{n=1}^{\infty}\frac{\lambda(m,n)}{mn}S(M,n;\ell r/m)\,\mathcal{U}^{+}\left(\frac{m^2n}{(\ell r)^{3}/N}\right)\bigg|\\
&\ll N^{\varepsilon}\frac{M^{3/2}t^{3/2}}{NP}\sum_{\ell\sim L}\sum_{r\sim N/MLt}\ell r\mathop{\sum\sum}_{m^2n<(\ell r)^3/N}\frac{|\lambda(m,n)|}{mn}\left(\frac{\ell r}{m}\right)^{1/2}\\
&\ll N^{\varepsilon}\frac{N^{3/2}}{PMt},
\end{split}
\end{equation*}
upon using Weil's bound. This bound turns out to be satisfactory for our purpose.

From now on we assume that $(p,\ell r)=1$. Then, an application of Voronoi summation \eqref{Voronoi0} to the $n$-sum yields
\begin{equation*}
\begin{split}
\sum_{n=1}^{\infty}\lambda(1,n)e\left(-\frac{np\bar{M}}{\ell r}\right)w\left(\frac{n}{N}\right)
=\ell r\sum_{\pm}\sum_{m|\ell r}\sum_{n=1}^{\infty}\frac{\lambda(n,m)}{mn}S(\bar{p}M,\pm n;\ell r/m)\,\mathcal{U}^{\pm}\left(\frac{m^2nN}{(\ell r)^{3}}\right).
\end{split}
\end{equation*}
Here the contribution from the terms with $m^2n\gg N^{\varepsilon}(\ell r)^3/N$ is negligibly small, and we can truncate the $(m,n)$-sum at $m^2n\ll N^{2+\varepsilon}P^3/M^3t^3$, at the cost of a negligible error. For those $m^2n\ll N^{2+\varepsilon}P^3/M^3t^3$, the result $|\Re \alpha_i|< 1/2$ for the Langlands parameter $\boldsymbol{\alpha}=(\alpha_1,\alpha_2,\alpha_3)$ of Jacquet and Shalika gives us the bound
\begin{equation*}
\mathcal{U}^{\pm}\left(\frac{m^2nN}{\ell^3 r^{3}}\right)\ll \sqrt{\frac{m^2nN}{\ell^3 r^{3}}},
\end{equation*}
while in general we have $y^{j}\mathcal{U}^{\pm,(j)}\left(y\right)\ll \sqrt{y}$.

After applying Voronoi summation, we have
\begin{equation*}
\begin{split}
\mathcal{F}_1=&\frac{M^{3/2}t^{3/2}}{NP^2}\sum_{\pm}\sum_{p\sim P}\bar{\chi}(p)p^{it}\sum_{\ell\sim L}\chi(\ell)\ell^{-it}\sum_{r=1}^{\infty}\chi(r)r^{-it}V\left(\frac{r}{Np/M\ell t}\right)\\
&\quad \ell r\mathop{\sum_{m|\ell r}\sum_{n\geq 1}}_{m^2n\ll N^{2+\varepsilon}P^3/M^3t^3}\frac{\lambda(n,m)}{mn}S(\bar{p}M,\pm n;\ell r/m)\,\mathcal{U}^{\pm}\left(\frac{m^2nN}{\ell^3 r^{3}}\right)+O\left(\frac{N^{3/2+\varepsilon}}{PMt}\right)\\
:=&\sum_{\pm}\mathcal{F}^{\pm}_1+O\left(\frac{N^{3/2+\varepsilon}}{PMt}\right).
\end{split}
\end{equation*}
Consider for example one of the two sums, $\mathcal{F}^{+}_1$. Pulling the $\ell$-sum inside the $(m,n)$-sum and applying the Cauchy--Schwarz inequality to $\mathcal{F}^{+}_1$, we obtain
 \begin{equation*}
\begin{split}
\mathcal{F}^{+}_1\asymp&\frac{(Mt)^{1/2}}{P}\bigg|\sum_{r=1}^{\infty}\chi(r)r^{-it}V\left(\frac{r}{NP/MLt}\right)\mathop{\sum\sum}_{\substack{m,n\geq 1\\ m^2n\ll N^{2+\varepsilon}P^3/M^3t^3}}\frac{\lambda(n,m)}{mn}\\
&\quad\quad\quad\sum_{\substack{\ell\sim L\\ m|\ell r}}\chi(\ell)\ell^{-it}\sum_{p\sim P}\bar{\chi}(p)p^{it}S(\bar{p}M,n;\ell r/m)\,\mathcal{U}^{+}\left(\frac{m^2nN}{\ell^3 r^{3}}\right)\bigg|\\
\ll&\frac{N^{1/2+\varepsilon}}{P^{1/2}L^{1/2}}\,\Sigma^{1/2},
\end{split}
\end{equation*}
where
\begin{equation}\label{diagonal+off-diagonal 1}
\begin{split}
\Sigma:=&\sum_{\substack{r\sim NP/MLt\\ (r,M)=1}}\mathop{\sum\sum}_{\substack{m,n\geq 1\\ m^2n\ll N^2P^3/M^3t^3}}\frac{1}{mn}\left|\sum_{\substack{\ell\sim L\\ m|\ell r}}\chi(\ell)\ell^{-it}\sum_{p\sim P}\bar{\chi}(p)p^{it}S(\bar{p}M,n;\ell r/m)\,\mathcal{U}^{+}\left(\frac{m^2nN}{\ell^3 r^{3}}\right)\right|^2,
\end{split}
\end{equation}
by noting that 
\begin{equation*}
\begin{split}
\mathop{\sum\sum}_{\substack{m,n\geq 1\\ m^2n\ll N^2P^3/M^3t^3}}\frac{|\lambda(n,m)|^2}{mn}\ll N^{\varepsilon}.
\end{split}
\end{equation*}

Now it remains to treat the sum $\Sigma$. Opening the square and interchanging the order of summations, we find
\begin{equation*}
\begin{split}
\Sigma \leq&\sum_{r\sim NP/MLt}\sum_{m<NP^{3/2}/M^{3/2}t^{3/2}}\frac{1}{m}\sum_{\substack{\ell_1\sim L\\ m| \ell_1 r}}\sum_{\substack{\ell_2\sim L\\ m| \ell_2 r}}\sum_{p_1\sim P}\sum_{p_2\sim P}\\
&\bigg|\sum_{n\ll  N^2P^3/m^2M^3t^3}\frac{1}{n}S(\bar{p_1}M,n;\ell_1 r/m)S(\bar{p_2}M,n;\ell_2 r/m)\,\mathcal{U}^{+}\left(\frac{m^2nN}{\ell_1^3\, r^{3}}\right)\overline{\mathcal{U}^{+}\left(\frac{m^2nN}{\ell_2^3 \,r^{3}}\right)}\bigg|.
\end{split}
\end{equation*}

Our next step is to apply Poisson summation to the $n$-sum. To this end, by a smooth dyadic partition of unity, one can insert an nonnegative smooth function $F(x)$ which is supported on, say $[1/2,3]$, and constantly $1$ on $[1,2]$, into the $n$-sum. 

Now for any
\begin{equation}\label{N_m1}
\mathcal{N}_m\ll N^2P^3/m^2M^3t^3,
\end{equation}
an application of Poisson summation with modulus $[\ell_1r/m,\ell_2 r/m]$ gives
\begin{equation*}
\begin{split}
&\sum_{n\geq 1}\frac{1}{n}S(\bar{p_1}M,n;\ell_1 r/m)S(\bar{p_2}M,n;\ell_2 r/m)\,\mathcal{U}^{+}\left(\frac{m^2nN}{\ell_1^3\, r^{3}}\right)\overline{\mathcal{U}^{+}\left(\frac{m^2nN}{\ell_2^3\, r^{3}}\right)}F\left(\frac{n}{\mathcal{N}_m}\right)\\
&=\frac{1}{[\ell_1,\ell_2] r/m}\sum_{n\in\mathbb{Z}}\mathcal{C}_{\ell_1,\ell_2}(n)\mathcal{T}(n,\ell_1,\ell_2),
\end{split}
\end{equation*}
where
\begin{equation}\label{C ell_1 ell_2}
\mathcal{C}_{\ell_1,\ell_2}(n)=\sum_{a([\ell_1,\ell_2]\frac{ r}{m})}S(\bar{p}_1M,a;\ell_1r/m)S(\bar{p}_2M,a;\ell_2r/m)e\left(\frac{an}{[\ell_1,\ell_2] r/m}\right),
\end{equation}
and
\begin{equation*}
\mathcal{T}(n,\ell_1,\ell_2)=\int_{\mathbb{R}}F(x)\,\mathcal{U}^{+}\left(\frac{m^2\mathcal{N}_mNx}{\ell_1^3\, r^{3}}\right)\overline{\mathcal{U}^{+}\left(\frac{m^2\mathcal{N}_mNx}{\ell_2^3\, r^{3}}\right)}e\left(-\frac{n\mathcal{N}_m}{[\ell_1,\ell_2] r/m}x\right)\frac{\mathrm{d}x}{x}.
\end{equation*}

By integrating by parts repeatedly, the integral $\mathcal{T}(n,\ell_1,\ell_2)$ is negligibly small if $|n|\geq  N^{\varepsilon}\frac{[\ell_1,\ell_2] r}{m\mathcal{N}_m}$. Therefore we can truncate the dual $n$-sum at $N^{\varepsilon}\frac{[\ell_1,\ell_2] r}{m\mathcal{N}_m}$, at the cost of a negligible error. Meanwhile in the range $|n|\ll N^{\varepsilon}\frac{[\ell_1,\ell_2] r}{m\mathcal{N}_m}$, we use the bounds $y^{j}\mathcal{U}^{+,(j)}\left(y\right)\ll \sqrt{y}$ to obtain
\begin{equation}\label{bound of T1} 
\mathcal{T}(n,\ell_1,\ell_2)\ll\frac{m^2\mathcal{N}_mN}{(\ell_1\ell_2)^{3/2}r^3}\ll \frac{m^2\mathcal{N}_mN}{(NP/Mt)^3}.
\end{equation} 

Let us also observe in particular that
\begin{equation*}
\mathcal{T}(n,\ell_1,\ell_2)\ll 1 \quad \mbox{and}\quad \mathcal{N}_1\ll N^2P^3/M^3t^3
\end{equation*}
for later convenience.

We arrive at
\begin{equation}\label{F1 after Cauchy1}
\begin{split}
\mathcal{F}_1
\ll&\frac{N^{1/2+\varepsilon}}{P^{1/2}L^{1/2}}\,\Omega^{1/2}+\frac{N^{3/2+\varepsilon}}{PMt},
\end{split}
\end{equation}
where
\begin{equation*}
\begin{split}
\Omega=\sum_{r\sim NP/MLt}\sum_{m<NP^{3/2}/M^{3/2}t^{3/2}}\sum_{\substack{\ell_1\sim L\\ m| \ell_1 r}}\sum_{\substack{\ell_2\sim L\\ m| \ell_2 r}}\sum_{p_1\sim P}\sum_{p_2\sim P}\sum_{|n|<[\ell_1,\ell_2] r/m\mathcal{N}_m}\frac{1}{[\ell_1,\ell_2] r}\,|\mathcal{C}_{\ell_1,\ell_2}(n)\mathcal{T}(n,\ell_1,\ell_2)|.
\end{split}
\end{equation*}

We have essentially square-root cancellation for the character sum $\mathcal{C}_{\ell_1,\ell_2}(n)$, defined in \eqref{C ell_1 ell_2}. The details of this calculation were carried out in \cite{HN17}. We have collected their results relevant to our present setting in Lemma \ref{ell_1 not ell_2}.

Bounding our sum \eqref{C ell_1 ell_2} using Lemma \ref{ell_1 not ell_2}, we get
\begin{equation}\label{character sum}
|\mathcal{C}_{\ell_1,\ell_2}(n)|\leq 2^{O(\omega(r))}\left(\frac{r}{m}\right)^{3/2} \frac{\ell_1\ell_2}{(\ell_1,\ell_2)^{1/2}}\frac{(\Delta,n,\ell_1r/m,\ell_2r/m)}{(n,\ell_1r/m,\ell_2r/m)^{1/2}},
\end{equation}
where 
\begin{equation*}
\Delta:=\frac{\bar{p_1}\ell_2^2-\bar{p_2}\ell_1^2}{(\ell_1,\ell_2)^2}M,
\end{equation*}
and $\bar{p_1}$ and $\bar{p_2}$ denote the multiplicative inverses of $p_1$ and $p_2$ modulo $\ell_1r/m$ and $\ell_2r/m$, respectively, and $\omega(r)$ denotes the number of distinct prime factors of $r$.

We write 
\begin{equation*}\Omega=\Omega_0+\Omega_1,
\end{equation*}
where $\Omega_0$ denotes the contribution from the terms $n=\Delta=0$, and $\Omega_1$ denotes the complement.
\begin{rem}
In fact, $\Omega_0$ is the diagonal contribution $(\ell_1,p_1)=(\ell_2,p_2)$ to the sum \eqref{diagonal+off-diagonal 1}, and $\Omega_1$ is the off-diagonal $(\ell_1,p_1)\neq (\ell_2,p_2)$ contribution.
\end{rem}

If $\Delta=0$, then $\bar{p_1}\ell_2^2-\bar{p_2}\ell_1^2=0$. Necessarily, $\ell_1=\ell_2:=\ell$ and $p_1=p_2:=p$. Under this condition,
\begin{equation*}
|C_{\ell,\ell}(n)|\leq 2^{O(\omega(r))}\left(\frac{\ell r}{m}\right)^{3/2}\left(n,\frac{\ell r}{m}\right)^{1/2}.
\end{equation*}
In particular, $|C_{\ell,\ell}(0)|\leq 2^{O(\omega(r))}\left(\frac{\ell r}{m}\right)^2$. 
Therefore,
\begin{equation}\label{estimate of omega01}
\begin{split}
\Omega_0\ll&  \sum_{r\sim NP/MLt}\sum_{\ell\sim L}\sum_{m|\ell r}\sum_{p\sim P}2^{O(\omega(r))}\frac{1}{\ell r}\left(\frac{\ell r}{m}\right)^2\,|\mathcal{T}(0,\ell,\ell)|
\ll \frac{N^{2+\varepsilon}P^3}{M^2t^2}.
\end{split}
\end{equation}
Here we have used the fact that $\omega(r)\ll \frac{\log r}{\log \log r}$.

Meanwhile for $\Omega_1$, we further write 
\begin{equation*}
\Omega_1=\Omega_{1a}+\Omega_{1b},
\end{equation*} 
where $\Omega_{1a}$ denotes the contribution coming from the $n\neq 0$ terms, and $\Omega_{1b}$ denotes the contribution of the zero frequency: $n=0,\, \Delta\neq 0$. Plugging the bounds \eqref{bound of T1}  and \eqref{character sum} in, we see that 
\begin{equation*}
\begin{split}
\Omega_{1a}\ll&\sum_{r\sim NP/MLt}\sum_{\ell_1\sim L}\sum_{\ell_2\sim L}\sum_{m|(\ell_1,\ell_2)r}\sum_{p_1\sim P}\sum_{p_2\sim P}\sum_{0\neq |n|<[\ell_1,\ell_2] r/m\mathcal{N}_m}\frac{|\mathcal{C}_{\ell_1,\ell_2}(n)|}{[\ell_1,\ell_2] r}\,|\mathcal{T}(n,\ell_1,\ell_2)|\\
\ll&N^{\varepsilon}\sum_{r\sim NP/MLt}\mathop{\sum_{\ell_1\sim L}\sum_{\ell_2\sim L}}_{\ell_1\neq \ell_2}\sum_{m|r}\mathop{\sum_{p_1\sim P}\sum_{p_2\sim P}}\sum_{0\neq |n|<\ell_1\ell_2 r/m\mathcal{N}_m}\frac{r^{1/2}}{m^{3/2}}\frac{(\Delta,n,r/m)}{(n,r/m)^{1/2}}\frac{m^2\mathcal{N}_mN}{(NP/Mt)^3}\\
\ll& N^{\varepsilon}\frac{NP}{MLt}L^2P^2\frac{NPL}{\mathcal{N}_1Mt}\left(\frac{NP}{MLt}\right)^{1/2}\frac{\mathcal{N}_1N}{(NP/Mt)^3}\\
\ll& N^{\varepsilon}(NMt)^{1/2}(PL)^{3/2}.
\end{split}
\end{equation*} 
Here in the second inequality above we have used the fact that the contribution from the case $\ell_1=\ell_2$ is comparably smaller.

Now we treat the case of $\Omega_{1b}$, which by our definition is
\begin{equation*}
\begin{split}
\Omega_{1b}=\sum_{r\sim NP/MLt}\sum_{m<NP^{3/2}/M^{3/2}t^{3/2}}\sum_{\substack{\ell_1\sim L\\ m| \ell_1 r}}\sum_{\substack{\ell_2\sim L\\ m| \ell_2 r}}\sum_{p_1\sim P}\sum_{p_2\sim P}\frac{1_{\Delta\neq 0}}{[\ell_1,\ell_2] r}\,|\mathcal{C}_{\ell_1,\ell_2}(0)\mathcal{T}(0,\ell_1,\ell_2)|.
\end{split}
\end{equation*}
A direct evaluation of $\mathcal{C}_{\ell_1,\ell_2}(0)$ from the definition \eqref{C ell_1 ell_2} shows that it vanishes unless $\ell_1=\ell_2:=\ell$. In the latter case we have
\begin{equation*}
\begin{split}
\mathcal{C}_{\ell,\ell}(0)=&\frac{\ell r}{m}\sumstar_{\beta(\ell r/m)}e\left(\frac{p_2-p_1}{\ell r/m}\beta\right)\\
\ll& \frac{\ell r}{m}\sum_{h|(\ell r/m,p_1-p_2)}h .
\end{split}
\end{equation*}
Recall for $\ell_1=\ell_2=\ell$, we have $\Delta=(\bar{p_1}-\bar{p_2})M$, where $\bar{p_1}$ and $\bar{p_2}$ are the multiplicative inverses of $p_1$ and $p_2$ modulo $\ell r/m$, respectively. As $\Delta\neq 0$, we have $p_1\neq p_2$.

We thus have
\begin{equation*}
\begin{split}
\Omega_{1b}\ll& N^{\varepsilon}\sum_{r\sim NP/MLt}\sum_{\ell\sim L}\sum_{m|\ell r}\sum_{h|\ell r/m}\frac{h}{m}\mathop{\sumsum_{p_1\neq p_2\sim P}}_{p_1\equiv p_2 (h)}1
\ll \frac{N^{1+\varepsilon}P^3}{Mt}.
\end{split}
\end{equation*}

This is dominated by the diagonal contribution $\Omega_0$ \eqref{estimate of omega01}, since $Mt<N$.

Hence we obtain the bound
\begin{equation}\label{estimate of omega1}
\begin{split}
\Omega=&\,\Omega_0+\Omega_{1a}+\Omega_{1b}\\
\ll& \frac{N^{2+\varepsilon}P^3}{M^2t^2}+N^{\varepsilon}(NMt)^{1/2}(PL)^{3/2}.
\end{split}
\end{equation}

Combining \eqref{F1 after Cauchy1} and \eqref{estimate of omega1}, we retrieve the bound on $\mathcal{F}_1$ in the following.
\begin{prop}\label{final bound for F1}
For any given $\varepsilon>0$,
\begin{equation*}
\begin{split}
\mathcal{F}_1\ll& \frac{N^{3/2+\varepsilon}P}{MtL^{1/2}}+N^{3/4+\varepsilon}(MtPL)^{1/4}.
\end{split}
\end{equation*}
\end{prop}
\begin{rem}
We will assume $L<P$, so that the term $O\left(\frac{N^{3/2+\varepsilon}}{PMt}\right)$ in \eqref{F1 after Cauchy1} is negligible.
\end{rem}

\section{The choices of the parameters $P$ and $L$}
Recall from Proposition \ref{final bound for F1}, one has
\begin{equation*}
\begin{split}
\mathcal{F}_1\ll& \frac{N^{3/2+\varepsilon}P}{MtL^{1/2}}+N^{3/4+\varepsilon}(MtPL)^{1/4},
\end{split}
\end{equation*}
while Proposition \ref{final bound for O} gives
\begin{equation*}
\begin{split}
\mathcal{O}\ll \frac{N^{1/2+\varepsilon}Mt}{P}+N^{\varepsilon}\frac{M^{3/2}t^{3/2}L}{P}.
\end{split}
\end{equation*}

Plugging these bounds into \eqref{connection1}, 
\begin{equation*}
\begin{split}
S(N)\ll& \frac{N^{3/2+\varepsilon}P}{MtL^{1/2}}+N^{3/4+\varepsilon}(MtPL)^{1/4}+\frac{N^{1/2+\varepsilon}Mt}{P}+N^{\varepsilon}\frac{M^{3/2}t^{3/2}L}{P}.
\end{split}
\end{equation*}

Substituting this into Lemma \ref{approximate FE} and noting that $(Mt)^{3/2-\delta}<N<(Mt)^{3/2+\varepsilon}$, one gets
\begin{equation*}
\begin{split}
L\left(\frac{1}{2}+it,\pi\otimes \chi \right)\ll& \frac{(Mt)^{1/2+\varepsilon}P}{L^{1/2}}+(Mt)^{5/8+\varepsilon}(PL)^{1/4}+\frac{(Mt)^{1+\varepsilon}}{P}+\frac{(Mt)^{3/4+\delta/2+\varepsilon}L}{P}
+(Mt)^{3/4-\delta/2+\varepsilon}\\
=&\frac{(Mt)^{5/8+\varepsilon}P^{1/4}}{L^{1/2}}\left(\frac{P^{3/4}}{(Mt)^{1/8}}+L^{3/4}\right)+(Mt)^{\varepsilon}\left(\frac{Mt}{P}+(Mt)^{3/4-\delta/2}\right),
\end{split}
\end{equation*}
upon assuming $L<(Mt)^{1/4-\delta/2}$.

Equate the first two terms by letting $L=P(Mt)^{-1/6}$ to get
\begin{equation*}
\begin{split}
L\left(\frac{1}{2}+it,\pi\otimes \chi \right)\ll& 
(Mt)^{7/12+\varepsilon}P^{1/2}+(Mt)^{1+\varepsilon}/P+(Mt)^{3/4-\delta/2+\varepsilon}.
\end{split}
\end{equation*}

Letting $P=(Mt)^{5/18}$,
\begin{equation}\label{latter case for t1}
\begin{split}
L\left(\frac{1}{2}+it,\pi\otimes \chi \right)\ll& 
(Mt)^{13/18+\varepsilon}+(Mt)^{3/4-\delta/2+\varepsilon}.
\end{split}
\end{equation}

Finally, by choosing $\delta=1/18$, \eqref{latter case for t1} implies that
\begin{equation*}
\begin{split}
L\left(\frac{1}{2}+it,\pi\otimes \chi \right)\ll& 
(Mt)^{3/4-1/36+\varepsilon}.
\end{split}
\end{equation*}
Note that with such choices, $L=(Mt)^{1/9}$ satisfies the assumption $L<(Mt)^{1/4-\delta/2}=(Mt)^{2/9}$.
Theorem \ref{Main theorem} follows.

%

\subsection*{Acknowledgements.}The author is most grateful to his thesis advisor, Professor Roman Holowinsky, for all his guidance and support through out the project. He thanks Zhi Qi and Runlin Zhang for several helpful discussions. The author would also like to express his thanks to Professor Paul D. Nelson for a few insightful comments and Alex Beckwith for useful suggestions which help improve the presentation of this article. Finally the author is very grateful to the referee for several very helpful comments and suggestions and in particular for suggesting a proof of Lemma \ref{bound of J}.

\bibliographystyle{abbrv}
\bibliography{ref}

\begin{thebibliography}{10}

\bibitem{Blo12}
V.~Blomer.
\newblock Subconvexity for twisted {$L$}-functions on {${\rm GL}(3)$}.
\newblock {\em Amer. J. Math.}, 134(5):1385--1421, 2012.

\bibitem{Blomer-Buttcane}
V.~Blomer and J.~Buttcane.
\newblock On the subconvexity problem for {$L$}-functions on {${\rm GL}(3)$}.
\newblock {\em Ann. Sci. \'{E}c. Norm. Sup\'{e}r. (to appear)}.

\bibitem{Blomer-Buttcane18}
V.~Blomer and J.~Buttcane.
\newblock Subconvexity for {$L$}-functions of non-spherical cusp forms on
  {${\rm GL}(3)$}.
\newblock {\em Acta Arith.}, 192(1):31--62, 2020.

\bibitem{BKY}
V.~Blomer, R.~Khan, and M.~Young.
\newblock Distribution of mass of holomorphic cusp forms.
\newblock {\em Duke Math. J.}, 162(14):2609--2644, 2013.

\bibitem{Conrey-Iwaniec}
J.~B. Conrey and H.~Iwaniec.
\newblock The cubic moment of central values of automorphic {$L$}-functions.
\newblock {\em Ann. of Math. (2)}, 151(3):1175--1216, 2000.

\bibitem{DFI}
W.~Duke, J.~Friedlander, and H.~Iwaniec.
\newblock Bounds for automorphic {$L$}-functions.
\newblock {\em Invent. Math.}, 112(1):1--8, 1993.

\bibitem{HN17}
R.~Holowinsky and P.~D. Nelson.
\newblock Subconvex bounds on {$\rm GL_3$} via degeneration to frequency zero.
\newblock {\em Math. Ann.}, 372(1-2):299--319, 2018.

\bibitem{Huang1}
B.~Huang.
\newblock Hybrid subconvexity bounds for twisted {$L$}-functions on {${\rm
  GL}(3)$}.
\newblock {\em Sci. China Math. (to appear)}.

\bibitem{Huxley}
M.~N. Huxley.
\newblock {\em Area, lattice points, and exponential sums}, volume~13 of {\em
  London Mathematical Society Monographs. New Series}.
\newblock The Clarendon Press, Oxford University Press, New York, 1996.
\newblock Oxford Science Publications.

\bibitem{Iw-Ko}
H.~Iwaniec and E.~Kowalski.
\newblock {\em Analytic number theory}, volume~53 of {\em American Mathematical
  Society Colloquium Publications}.
\newblock American Mathematical Society, Providence, RI, 2004.

\bibitem{Jacquet-Shalika81}
H.~Jacquet and J.~A. Shalika.
\newblock On {E}uler products and the classification of automorphic
  representations. {I}.
\newblock {\em Amer. J. Math.}, 103(3):499--558, 1981.

\bibitem{Kiral-Petrow-Young}
E.~M. Kiral, I.~Petrow, and M.~P. Young.
\newblock Oscillatory integrals with uniformity in parameters.
\newblock {\em J. Th\'{e}or. Nombres Bordeaux}, 31(1):145--159, 2019.

\bibitem{Li1}
X.~Li.
\newblock Bounds for {${\rm GL}(3)\times {\rm GL}(2)$} {$L$}-functions and
  {${\rm GL}(3)$} {$L$}-functions.
\newblock {\em Ann. of Math. (2)}, 173(1):301--336, 2011.

\bibitem{MSY}
M.~McKee, H.~Sun, and Y.~Ye.
\newblock Improved subconvexity bounds for {$GL(2)\times GL(3)$} and {$GL(3)$}
  {$L$}-functions by weighted stationary phase.
\newblock {\em Trans. Amer. Math. Soc.}, 370(5):3745--3769, 2018.

\bibitem{Michel-Venkatesh}
P.~Michel and A.~Venkatesh.
\newblock The subconvexity problem for {${\rm GL}_2$}.
\newblock {\em Publ. Math. Inst. Hautes \'Etudes Sci.}, (111):171--271, 2010.

\bibitem{Miller01}
S.~D. Miller.
\newblock On the existence and temperedness of cusp forms for {${\rm
  SL}_3({\Bbb Z})$}.
\newblock {\em J. Reine Angew. Math.}, 533:127--169, 2001.

\bibitem{Mil1}
S.~D. Miller.
\newblock Cancellation in additively twisted sums on {${\rm GL}(n)$}.
\newblock {\em Amer. J. Math.}, 128(3):699--729, 2006.

\bibitem{Miller-Schmid}
S.~D. Miller and W.~Schmid.
\newblock Automorphic distributions, {$L$}-functions, and {V}oronoi summation
  for {${\rm GL}(3)$}.
\newblock {\em Ann. of Math. (2)}, 164(2):423--488, 2006.

\bibitem{Molteni}
G.~Molteni.
\newblock Upper and lower bounds at {$s=1$} for certain {D}irichlet series with
  {E}uler product.
\newblock {\em Duke Math. J.}, 111(1):133--158, 2002.

\bibitem{Mun2}
R.~Munshi.
\newblock Twists of {$\rm GL(3)$} {$L$}-functions.
\newblock {\em ArXiv preprint (2016), arXiv:1604.08000}.

\bibitem{Munshi2014}
R.~Munshi.
\newblock The circle method and bounds for {$L$}-functions---{I}.
\newblock {\em Math. Ann.}, 358(1-2):389--401, 2014.

\bibitem{Mun4}
R.~Munshi.
\newblock The circle method and bounds for {$L$}-functions---{III}:
  {$t$}-aspect subconvexity for {$GL(3)$} {$L$}-functions.
\newblock {\em J. Amer. Math. Soc.}, 28(4):913--938, 2015.

\bibitem{Mun1}
R.~Munshi.
\newblock The circle method and bounds for {$L$}-functions---{IV}:
  {S}ubconvexity for twists of {$\rm GL(3)$} {$L$}-functions.
\newblock {\em Ann. of Math. (2)}, 182(2):617--672, 2015.

\bibitem{Mun3}
R.~Munshi.
\newblock The circle method and bounds for {$L$}-functions, {II}:
  {S}ubconvexity for twists of {${\rm GL}(3)$} {$L$}-functions.
\newblock {\em Amer. J. Math.}, 137(3):791--812, 2015.

\bibitem{Nunes}
R.~M. Nunes.
\newblock Subconvexity for {${\rm GL}(3)$} {$L$}-functions.
\newblock {\em ArXiv preprint (2017), arXiv:1703.04424}.

\bibitem{Qi2}
Z.~Qi.
\newblock Theory of fundamental {B}essel functions of high rank.
\newblock {\em Mem. Amer. Math. Soc. (to appear)}.

\bibitem{Schumacher}
R.~Schumacher.
\newblock Subconvexity for {${\rm GL}_3(\mathbb{R})$} {$L$}-functions via
  integral representations.
\newblock {\em ArXiv preprint (2020), arXiv:2004.06791}.

\bibitem{Sun2017}
Q.~Sun.
\newblock Hybrid bounds for twists of {$GL(3)$} {$L$}-functions.
\newblock {\em Publ. Mat.}, 64(1):75--102, 2020.

\bibitem{Sun-Zhao17}
Q.~Sun and R.~Zhao.
\newblock Bounds for {$\rm GL_3$} {$L$}-functions in depth aspect.
\newblock {\em Forum Math.}, 31(2):303--318, 2019.

\bibitem{Young17}
M.~P. Young.
\newblock Weyl-type hybrid subconvexity bounds for twisted {$L$}-functions and
  {H}eegner points on shrinking sets.
\newblock {\em J. Eur. Math. Soc. (JEMS)}, 19(5):1545--1576, 2017.

\end{thebibliography}
\end{document}